\documentclass[11pt,a4paper]{article}
\usepackage[utf8]{inputenc}

\usepackage[hidelinks]{hyperref}
\usepackage{multicol}
\usepackage{amsmath, amsfonts, amssymb,mathtools,mathrsfs}
\usepackage{amsthm}
\usepackage{graphicx}
\graphicspath{ {images/} }
\usepackage{pst-all,pst-infixplot,pst-math,pst-fractal,pst-solides3d}
\usepackage{pict2e}

\newtheorem{thm}{Theorem}[section]
\newtheorem{prop}{Proposition}[section]

\newtheorem{lemma}{Lemma}[section]
\newtheorem{corollary}{Corollary}[section]

\newtheorem{defi}{Definition}[section]
\newtheorem{rem}{Remark}[section]

\DeclareMathOperator*{\essinf}{ess\,inf}

\newenvironment{Assumptions}
{
\setcounter{enumi}{0}

\begin{enumerate}}
{\end{enumerate} }

\title{On the well-posedness of a Hamilton-Jacobi-Bellman equation with transport noise}

\author{\textsc{Neeraj Bhauryal$^{a}$\footnote{neeraj.bhauryal@gmail.com}, Ana Bela Cruzeiro$^{b}\footnote{ana.cruzeiro@tecnico.ulisboa.pt}$, Carlos Oliveira$^{c,d}$\footnote{carlos.m.d.s.oliveira@ntnu.pt}}\\
 $^{a}$\textit{GFMUL, Universidade de Lisboa}\\
 $^{b}$\textit{GFMUL and Dep. de Matemática, Instituto Superior Técnico} \\
 \textit{Av. Rovisco Pais, 1049-001 Lisboa, Portugal} \\
 $^c$\textit{Department of Industrial Economics and Technology Management,}\\ \textit{Norwegian University of Science and Technology, 7491 Trondheim, Norway}\\
 $^{d}$\textit{ISEG - School of Economics and Management, Universidade de Lisboa}\\
 \textit{Rua do Quelhas 6, Lisboa 1200-781, Portugal}\\
 \textit{and REM-Research in Economics and Mathematics, CEMAPRE}}

\begin{document}
\maketitle
\date{}
\begin{abstract}
        In this paper we consider the following non-linear stochastic partial differential equation (SPDE): 
\begin{align*}
\begin{cases}
     \mathrm{d}u(s,x)=\sum^n_{i=1} \mathscr{L}_i u(s,x)\circ \mathrm{d}W_i(s)+\left(V(x)+\mu\Delta u(s,x)-\frac{1}{2}\vert\nabla u(s,x)\vert^2\right)\mathrm{d}s, \quad &\text{in } (0,T)\times \mathbb{T}^n,\\
    u(0,x)=u_0(x), & \text{on } \mathbb{T}^n,
    \end{cases}
\end{align*}
where $\mathbb{T}^n$ is the $n$-dimensional torus, the functions $u_0, V: \mathbb{T}^n \to \mathbb{R}$ are given and $\{\mathscr{L}_i\}_i$ is a collection of first order linear operators. This can be seen as a Cauchy problem for a Hamilton-Jacobi-Bellman equation with transport noise in any space dimension. We introduce the concept of a strong solution from the realm of PDEs and establish the existence and uniqueness of maximal solutions (strong solutions upto a stopping time). Moreover, for a particular class of $\{\mathscr{L}_i\}_i$ we establish global well-posedness of strong solutions. The proof relies on studying an associated truncated version of the original SPDE and showing its global well-posedness in the class of strong solutions.
\end{abstract}

\section{Introduction}
The field of SPDEs has been an active area of research due to their vast applications in Applied Mathematics, Biology, Engineering, Economics, Physics among others.
 The recent years have witnessed a surge in articles  interested in stochastic perturbations of partial differential equations (PDEs), since the study of random effects is crucial while crafting effective mathematical models  
to describe complex phenomena. There are a number of ways through which randomness can enter into this type of models, and we are particularly interested in stochastic forcing of a class of PDEs arising from optimal control.

A stochastic Hamilton-Jacobi-Bellman (HJB) equation is a SPDE, which appears associated with stochastic optimal control problems where the action functional is itself random. Peng pioneered the topic with the paper \cite{peng1992stochastic}. Several others have been contributing to this field by studying different action functionals, both in the field of Financial Mathematics (\cite{englezos2009utility,el2018consistent,graewe2015non}) and Physics (\cite{anabela}, \cite{gay}).

In this article, we are interested in a special class of stochastic HJB equations, namely stochastic HJB equations with a multiplicative noise given by the gradient, a fact that results in a number of technical challenges.   The question of well-posedness for SPDEs with gradient noise has been considered by many authors in recent years.  Such equations have been considered in  \cite{Lions} and \cite{lions1998fully} in connection with a pathwise generalized optimal control problem where the underlying stochastic flow is driven by two independent Brownian motions and the action functional is defined taking the expectation with respect to one of them. Often, there are no global smooth solutions (even in the deterministic case) for HJB equations. The notion of stochastic viscosity solutions for such equations have been proposed firstly by Lions and Souganidis (cf. \cite{Lions, lions1998fully,lions2000fully}) and developed by others (cf. \cite{Buckdahn,BUCKDAHN2001205,qiu2018viscosity, seeger2018perron}). For an overview of the pathwise theory of viscosity solutions we refer the reader to Souganidis \cite{Souganidis}. For first order parabolic SPDEs, Kunita \cite{kunita} proved the existence and uniqueness of local solution using stochastic characteristic curves. For a linear transport equation with gradient noise, a significant work is established in \cite{gubinelli}, wherein the authors proved existence and uniqueness of weak $L^\infty$ solutions. Recently, for nonlinear scalar conservation laws, the results were extended by the authors in \cite{gess} to the framework of kinetic solutions. 
 
In this work, we show existence of maximal solution in $C([0,T]; H^k(\mathbb{T}^n))$ a.s. to the following Cauchy problem
\begin{align}\label{HJB}
\begin{cases}
    \mathrm{d}u(s,x)=\sum^n_{i=1}\mathscr{L}_i u(s,x)\circ \mathrm{d}W_i(s)+\left(V(x)+\mu\Delta u(s,x)-\frac{1}{2}\vert\nabla u(s,x)\vert^2\right)\mathrm{d}s, \quad &\text{in } (0,T)\times \mathbb{T}^n,\\
    u(0,x)=u_0(x), & \text{on } \mathbb{T}^n,
    \end{cases}
\end{align}
where $\mathbb{T}^n$ is the $n$-dimensional flat torus, the functions $u_0, V: \mathbb{T}^n \to \mathbb{R}$ are given (see Section \ref{sec:tech} for the assumptions) and $\mu$ is a given positive constant. Here, $\mathscr{L}_i$ denotes the first order linear operator defined as $\mathscr{L}_i u:= a_i(x)\frac{\partial}{\partial x_i}u+b_i(x)u$ for $i=1,2,\cdots,n$. We consider a stochastic basis $(\Omega, \mathcal{F}, \mathbb{P}, (\mathcal{F}_t)_{t \ge 0})$, where $(\Omega, \mathcal{F}, \mathbb{P})$ is a probability space, $(\mathcal{F}_t)_{t \ge 0}$ is a complete filtration and $W$ is an $n-$dimensional Brownian motion defined on $\Omega$. The symbol  $\circ$ denotes that the stochastic integral is understood in the Stratonovich sense.\\ 
The Cauchy problem \eqref{HJB} is a generalization of the following 
\begin{align}\label{HJBc}
\begin{cases}    \mathrm{d}u(s,x)=-\sqrt{\nu}\nabla u(s,x)\circ \mathrm{d}W(s)+\left(V(x)+\mu\Delta u(s,x)-\frac{1}{2}\vert\nabla u(s,x)\vert^2\right)\mathrm{d}s, \quad &\text{in } (0,T)\times \mathbb{T}^n,\\
    u(0,x)=u_0(x), & \text{on } \mathbb{T}^n,
    \end{cases}
\end{align}
which is related with a pathwise control problem introduced in \cite{lions1998fully}. They considered the following stochastic differential equation
 \begin{align*}
     \mathrm{d}Z_t =b(u_t,Z_t)\mathrm{d}t+\Tilde{\sigma}(Z_t,u_t) \mathrm{d}W_t+\sigma(Z_t)\circ\mathrm{d}B_t,~~
     Z_0=x,
 \end{align*}
 where $W_t$ and $B_t$ are two independent Brownian motions on $(\Omega, \mathcal{F}, \mathbb{P})$, $u$ is a control process and denote by $\{\mathcal{F}^B_t\}_{t \ge 0}$ the filtration generated by $B$. We are interested in a particular case where $\widetilde{\sigma}$ and $\sigma$ are postivie constants, denoted by $\mu$ and $\nu$ respectively and the objective is to minimize the random action functional $J_{t,x}$ defined as 
 \begin{align*}
     J_{t,x}(u)= E\left[\int^t_0 \left( \frac12| u_s|^2+V(Z_s)\right)\,ds+S(Z_t)\bigg| \mathcal{F}^B_t\right]
 \end{align*}
 over the set of admissible controls $\mathcal{U}$, for some potential functions $V$ and $S$. The corresponding value process is defined as 
 $$U(t,x):= \essinf_{u \in \mathcal{U}} J_{t,x}(u), $$
which satisfies \eqref{HJBc}, according to Lions and Souganidis \cite{lions1998fully} (see Theorem 2.1). To the best of our knowledge, this type of stochastic equations have not been studied yet in the classical sense.

When it is possible to differentiate the solutions of stochastic HJB equations, which is the case in our results, we obtain a stochastic Burgers equation with gradient noise. Stochastic Burgers equations have been studied intensively, in various contexts, but mostly as random perturbations of the deterministic Burgers equation by a Brownian motion (additive noise, sometimes multiplicative but not often of  gradient type). 
Concerning  the equations we study, it turns out that they are
 important for turbulence in fluid dynamics. In particular they appear when decomposing Lagrangian flows into a slow large-scale mean and a rapidly fluctuating small-scale part \cite{cotter}.
 On the other hand they are related to pathwise variational principles, as said before.
 An existence theorem for the corresponding Cauchy problem in one dimension was proved in \cite{alonso}.
 
The remainder of the paper is structured as follows: In Section, \ref{sec:tech} we describe the technical framework, establish the key notation and introduce the assumptions before finally stating the main result of this article. In Section \ref{sec:truncated}, we introduce a truncated version of the original problem and prove its well-posedness. This is done by considering a hyper-regularised truncated equation and proving global existence of mild solutions.  We then demonstrate the required compactness methods and techniques to construct a limiting solution from this hyper-regularised equation. In Section \ref{sec:main} we first establish a key maximum principle, which helps us in  establishing the well-posedness theory for the problem under consideration \eqref{HJB}. Finally, Appendix \ref{app} recapitulates a few existing results about analytical semigroups and stochastic analysis.  

\section{Technical framework and statement of the main results}\label{sec:tech}
We start this section by introducing some notation and definitions, which will be necessary throughout the paper.
In what follows, the letter $C$ denotes various generic constants, which may differ from line to line, but we prefer to keep the notation unchanged. Whenever a stopping time is defined in this article, we use the convention $\inf \phi = \infty$, where $\phi$ denotes the empty set. The expression $``a\lesssim b"$ will be used to represent that for some $C>0$, $a \le C b$. Let us begin by stating some notations and spaces used throughout the article. Sobolev spaces are defined as $$W^{k,p}(\mathbb{T}^n):=\{ f \in L^p(\mathbb{T}^n) : (I- \Delta)^{k/2}f \in L^p(\mathbb{T}^n)\}$$
for any $k\ge 0$ and $p\ge 1$, equipped with the norm $\|f \|_{W^{k,p}}=  \|(I-\Delta)^{k/2}f\|_{L^p}$, where the term $(I-\Delta)^{k/2}f$ is the function whose Fourier transform is given by $(1+|\xi|^2)^{k/2}\hat{f}(\xi)$, and $\hat{f}$ denotes the Fourier transform of the function $f$. For the particular case $p=2$, we will use the notation $H^k(\mathbb{T}^n)= W^{k,2}(\mathbb{T}^n)$, with the dual being denoted by $H^{-k}(\mathbb{T}^n)$. The fractional Laplacian of order $0<k<2$,  $(-\Delta)^{k/2}$ is denoted by $\Lambda^k$. We also recall the space with fractional power in time, for fixed $p>1$ and $0<\alpha<1$, we define
$$
\mathscr{W}^{\alpha, p}([0, T] ; X)=\left\{u \in L^p([0, T] ; X): \int_0^T \int_0^T \frac{\|u(t)-u(s)\|_X^p}{|t-s|^{1+p \alpha}} \mathrm{d} t \mathrm{~d} s<\infty\right\} ,
$$
for a Banach space $X$ and
we endow this space with the following norm
$$
\|u\|_{\mathscr{W}^{\alpha, p}([0, T] ; X)}:=\left(\int_0^T\|u(t)\|_X^p \mathrm{~d} t+\int_0^T \int_0^T \frac{\|u(t)-u(s)\|_X^p}{|t-s|^{1+p \alpha}} \mathrm{d} t \mathrm{~d} s \right)^{1/p}.
$$
The objective of this paper is to establish the existence and uniqueness of classical solutions for the Cauchy problem \eqref{HJB}. Although the concept of (stochastic) viscosity solutions is  commonly employed in (stochastic) optimal control (see for instance \cite{Buckdahn,fleming2006controlled,Lions}) to prove existence and uniqueness of solution to the respective (stochastic) HJB equation, we opt for a different concept of solution in our analysis. We shall work with the following assumptions:
\begin{Assumptions}
    \item \label{A1} The initial data $u_0$ is deterministic and belongs to $H^k(\mathbb{T}^n)$.
    \item \label{A3} The coefficients $a_i,b_i \in W^{k+1,\infty}(\mathbb{T}^n)$ for each $i=1,2,\cdots, n$ and $V\in H^k(\mathbb{T}^n)$. 
\end{Assumptions}
We begin by stating some different notions of solution that would be required in the analysis.

\begin{defi}[Weak solution]
A pair $(\tau,u)$ is called a weak solution of \eqref{HJB}, where $u:[0, \tau] \times \mathbb{T}^n \times \Omega \rightarrow \mathbb{R}$, $\tau : \Omega \to [0,\infty]$ is a stopping time, if for every $\varphi \in C^\infty(\mathbb{T}^n)$, $\langle u(t\wedge \tau), \varphi \rangle_{L^2}$ is adapted to $(\mathcal{F}_t)_{t\ge 0}$  and satisfies the following 
$$
\begin{aligned}
\langle u(\tau'),\varphi \rangle_{L^2}&=\langle u(0), \varphi \rangle_{L^2} + \sum^n_{i=1}\int^{\tau'}_0 \langle u, \mathscr{L}^*_i\varphi \rangle_{L^2} \circ \mathrm{d}W_i(s)+\int_0^{\tau'}\left(\langle V,\varphi \rangle_{L^2} +\mu \langle u,\Delta \varphi \rangle_{L^2}\right)\mathrm{d}s\\
&- \int^{\tau'}_0 \frac12  \left\langle|\nabla u|^2,\varphi \right \rangle_{L^2}\mathrm{d}s 
\end{aligned}
$$
for all finite stopping time $\tau' \le \tau$.   
\end{defi}
\begin{defi}[Local solution]
Let $\tau:\Omega\to[0,\infty]$ be a stopping time. A random variable $u: \Omega \times[0, \tau]\times\mathbb{T}^n \rightarrow \mathbb{R}$ is called a local solution of \eqref{HJB} if it has trajectories of class $C\left([0, \tau] ; H^k\left(\mathbb{T}^n\right)\right)$ for $k> n/2$, such that $u(t \wedge \tau)$ is adapted to $\left(\mathcal{F}_{t}\right)_{t \geq 0}$, and the following holds
$$
\begin{aligned}
u_{\tau^{\prime}}-u_{0} &+\frac12\int_{0}^{\tau^{\prime}} |\nabla u|^2 \mathrm{d}s+\sum^n_{i=1}\int_{0}^{\tau^{\prime}} \mathscr{L}_i u \circ \mathrm{d}W_i(s) \\
&= \int^{\tau'}_0 \left( V(x) +\mu \Delta u\right)\mathrm{d}s
\end{aligned}
$$
for finite stopping times $\tau^{\prime} \leq \tau$.
\end{defi}
\begin{defi}[Maximal solution]
Let $\tau_{\text{max}}: \Omega \to [0,\infty]$ be a stopping time. A maximal solution of \eqref{HJB} is a pair $(\tau_{\text{max}},u)$, where $u:\Omega \times[0, \tau_{\text{max}}) \times \mathbb{T}^n\to \mathbb{R}$ is a random variable such that:
\begin{itemize}
    \item[-] $\mathbb{P}\left(\tau_{\max }>0\right)=1$, and $\tau_{\max }=\lim _{n \rightarrow \infty} \tau_n$ for an increasing sequence of stopping times $\left\{\tau_n\right\}_{n=1}^{\infty}$, where each $(\tau_n,u)$ is a local solution for $n \in \mathbb{N}$.
\item[-] If $\left(\tau^{\prime}, u^{\prime}\right)$ is another pair satisfying the above two conditions and $u^{\prime}=u$ on $\left[0, \tau^{\prime} \wedge \tau_{\max }\right)$, then $\tau^{\prime} \leq \tau_{\max }$, almost surely.
\end{itemize}
\end{defi}
\begin{defi}[Global classical solution]
A global classical solution of \eqref{HJB} is maximal solution of \eqref{HJB} such that $\tau_{\max }=\infty$, almost surely.
\end{defi}
We are now ready to state the main result of this article:
\begin{thm}\label{uniqueness}
Assume \ref{A1}-\ref{A3}; then there exists a unique pathwise maximal solution $u:\Omega \times [0,\infty) \times \mathbb{T}^n \to \mathbb{R}$ of the stochastic HJB equation \eqref{HJB}  with values in $H^k(\mathbb{T}^n)$ for $k>\frac{n}{2}+2$. Furthermore, there exists a  global classical solution for a particular class of $\mathscr{L}_i$ mentioned in Section \ref{sec:main}.
\end{thm}

% \begin{rem}
%     Our results can be extended to the case of a compact manifold $M$ without boundary instead of the Torus $\mathbb{T}^n$.
% \end{rem}
\begin{rem}
    For $\mu=0$ in  Eq. \eqref{HJB}, we only get well-posedness of a maximal solution $(\tau_{\text{max}},u)$ in $H^k(\mathbb{T}^n)$.
\end{rem}

\section{Well-posedness of the truncated equation}\label{sec:truncated}
We begin this section by introducing a truncated version of the stochastic HJB equation and first prove the uniqueness of global solution for the truncated equation in Subsection \ref{subsec:uniqueness}. The existence of solution to the truncated equation is then proven using the regularization with a higher order operator in Subsection \ref{subsec:existence}.\\
The investigation into the existence of solution for Eqn. \eqref{HJB} is related with the study of the following associated truncated equation
\begin{align}\label{truncated}
    \mathrm{d}u_r(s,x)= \sum^n_{i=1} \mathscr{L}_i u_r(s,x)\circ \mathrm{d}W_i(s)+\left(V(x)+\mu\Delta u_r(s,x)-\frac12\theta_r(\|\nabla u_r(s)\|_\infty)\vert\nabla u_r(s,x)\vert^2\right)\mathrm{d}s
\end{align}
with the initial data $u_0$,
where for some $r>0$, $\theta_r:[0,\infty) \to [0,1]$ is a smooth non-increasing function defined as
\begin{align*}
    \theta_r(x)=
    \begin{cases}
         1, \quad x\in [0, r],\\
         0, \quad x \in [r+1, \infty),
    \end{cases}
\end{align*}
and $\theta_r$ is some non-negative polynomial in $[r,r+1]$. We also introduce the stopping time important for our analysis, namely $\tau_r =\inf \{ t \ge 0 : \|u_r(t,\cdot)\|_{H^k} \ge \frac rC\}$, where $C$ is a constant in the Sobolev inequality  $\|\nabla u_r\|_\infty \le C\| u_r \|_{H^k(\mathbb{T}^n)} $. It is common (see \cite{alonso-leon},\cite{crisan},\cite{glatt}) to employ a cut-off in the nonlinear term to prove the existence and uniqueness of maximal solutions to a truncated equation like \eqref{truncated}.

We shall now state the main well-posedness result for the global solutions of \eqref{truncated}, which will be the foundation to prove well-posedness results for \eqref{HJB}.
\begin{thm}\label{WP-truncated}
For given $u_0\in H^k(\mathbb{T}^n)$ for some $k>\frac{n}{2}+2$ and $r>0$ fixed, there exists a unique global solution in $H^k(\mathbb{T}^n)$ for \eqref{truncated}.
\end{thm}
The proof the theorem is split in two parts, where we prove uniqueness of global solutions and existence of global solutions in the subsections \ref{subsec:uniqueness} and \ref{subsec:existence} respectively.

\subsection{Global Uniqueness of Solution to Truncated Equation}\label{subsec:uniqueness}
% We shall first establish the uniqueness of global solutions for \eqref{truncated}, for this we follow the similar strategy as for the proof of Proposition \eqref{uniquess-local}. However, we need to perform $H^k$ estimates and therefore the proof is more involved.
Denote by $u_1$ and $u_2$ the two global solutions of \eqref{truncated} defined on $\Omega \times [0, \infty) \times \mathbb{T}^n$ with the same initial data $u_0$, define the stopping time  $\tau_N= \inf \{ t \ge 0 : \| u_1\|_{H^k}+\|u_2\|_{H^k} \ge N\}$. We prove the uniqueness in the $H^k$ topology i.e. $\|u_1(s)-u_2(s)\|_{H^k}=0$ for all $s\in [0,T]$. First, notice that  $\tau_N \to \infty$ as $N \to \infty$ because $u_i \in  L^\infty([0,T];H^k(\mathbb{T}^n))\,$ a.s. from the \textit{a priori} estimate proven in the next section, therefore it suffices to show that $\|u_1(s)-u_2(s)\|_{H^k}=0$ on $[0, \tau_N \wedge T] \,\,\mathbb{P}$-a.s. for all $N$. Denote by $u:=u_1-u_2$, then $u$ satisfies the following equation
\begin{align*}   
\mathrm{d}u(s,x)&-\sum^n_{i=1}\mathscr{L}_iu(s,x) \mathrm{d}W_i(s) -\mu \Delta u(s,x)\mathrm{d}s-\frac12\sum^n_{i=1}\mathscr{L}^2_iu(s,x)\mathrm{d}s\\
    &=-\frac12\left(\theta_r(\|\nabla u_1(s)\|_\infty)\vert\nabla u_1(s,x)\vert^2 -\theta_r(\|\nabla u_2(s)\|_\infty)\vert\nabla u_2(s,x)\vert^2 \right)\mathrm{d}s.
\end{align*}
Applying It\^o's formula to $\| \cdot\|_{L^2}$, we get 
\begin{align*}
    &\frac12 \mathrm{d}\| u(s)\|^2_{L^2}- \sum^n_{i=1} \langle \mathscr{L}_i u(s), u(s)\rangle_{L^2}\mathrm{d}W_i(s)-\frac12\sum^n_{i=1}\left( \langle \mathscr{L}_i u(s),\mathscr{L}_i u(s)\rangle_{L^2}+ \langle \mathscr{L}^2_i u(s), u(s)\rangle_{L^2}\right) \mathrm{d}s \\
&=\left(\mu \langle \Delta u(s),u(s) \rangle_{L^2}+\frac12\left\langle \theta_r(\|\nabla u_1(s)\|_\infty)\vert\nabla u_1(s)\vert^2 -\theta_r(\|\nabla u_2(s)\|_\infty)\vert\nabla u_2(s)\vert^2 , u(s)\right\rangle_{L^2}\right)\mathrm{d}s.
\end{align*}
Making use of the fact $\langle \Delta u, u\rangle_{L^2} \le 0$ and Lemma \ref{operator}, we write the above equation as
\begin{align}\label{ito-l2}
     &\mathrm{d} \| u(s)\|^2_{L^2}- 2\sum^n_{i=1}\langle \mathscr{L}_i u(s), u(s)\rangle_{L^2}\mathrm{d}W_i(s)-\|u(s)\|^2_{L^2}\mathrm{d}s \notag\\
     &\qquad\lesssim  \left\langle \theta_r(\|\nabla u_1(s)\|_\infty)\vert\nabla u_1(s,x)\vert^2-\theta_r(\|\nabla u_2(s)\|_\infty)\vert\nabla u_2(s,x)\vert^2 , u(s)\right\rangle_{L^2}\mathrm{d}s.
\end{align}
To handle the nonlinearty on the RHS of \eqref{ito-l2}, we write it as
\begin{align*}
    &\left\langle \theta_r(\|\nabla u_1\|_\infty)\vert\nabla u_1\vert^2 -\theta_r(\|\nabla u_2\|_\infty)\vert\nabla u_2\vert^2, u\right\rangle_{L^2}\\
    &=\left\langle \left(\theta_r(\|\nabla u_1\|_\infty) -\theta_r(\|\nabla u_2\|_\infty)\right)\vert\nabla u_1\vert^2, u\right\rangle_{L^2}+ \left\langle \theta_r(\| \nabla u_2\|_\infty)\nabla u_1 \cdot \nabla u, u \right\rangle_{L^2}\\
    &\qquad+\left\langle \theta_r(\| \nabla u_2\|_\infty)\nabla u_2 \cdot \nabla u, u \right\rangle_{L^2}\\
    &\le |\theta_r(\|\nabla u_1\|_\infty) -\theta_r(\|\nabla u_2\|_\infty)| |\nabla u_1|^2_\infty \|u\|_{L^2}+|\left\langle \nabla u_1 \cdot \nabla u, u \right\rangle_{L^2}|+|\left\langle \nabla u_2 \cdot \nabla u, u \right\rangle_{L^2}|\\
  &  \lesssim N^2\|u\|_{H^k}\|u\|_{L^2}+\|u_1\|_{H^k}\|u\|^2_{L^2}+\|u_2\|_{H^k}\|u\|^2_{L^2}\\
  & \lesssim (1+ \|u_1\|_{H^k}+\|u_2\|_{H^k})\|u\|^2_{H^k}
\end{align*}
where in the penultimate inequality  we use \begin{align}\label{IBP-unique}
\langle \nabla u_i\cdot \nabla u,u \rangle_{L^2} = \frac12\langle - \Delta u_i, u^2\rangle_{L^2} \lesssim \|\Delta u_i \|_{L^\infty} \|u\|^2_{L^2} \lesssim \|u_i\|_{H^k} \|u\|^2_{L^2}, 
\end{align}
for $i=1,2$ and $k>2$ and the fact that $\theta_r \le 1$  and $\theta_r$ is Lipschitz continuous as follows
\begin{align*}
    |\theta_r(\|\nabla u_1\|_\infty) -\theta_r(\|\nabla u_2\|_\infty)| \le C| \|\nabla u_1\|_\infty- \| \nabla u_2\|_\infty| \le C \|\nabla u\|_\infty \le C \| u\|_{H^k}.
\end{align*}
Substituting this in \eqref{ito-l2} we get 
\begin{align}\label{uniquel2}
\mathrm{d}\|u(s)\|^2_{L^2}- 2\sum^n_{i=1}\langle \mathscr{L}_i u(s), u(s)\rangle_{L^2}\mathrm{d}W_i(s) \lesssim (1+ \|u_1\|_{H^k}+\|u_2\|_{H^k})\|u(s)\|^2_{H^k}\mathrm{d}s.
\end{align}
Next, we derive an estimate for the semi-norm; recall that $\Lambda^k = (-\Delta)^{k/2}$ and consider
\begin{align*}
   & \frac12 \mathrm{d}\| \Lambda^k u\|^2_{L^2}- \sum^n_{i=1}\langle  \Lambda^k \mathscr{L}_i u,\Lambda^k u\rangle_{L^2}\mathrm{d}W_i(s)-\frac12\sum^n_{i=1}\left( \langle \Lambda^k \mathscr{L}_i u(s),\Lambda^k \mathscr{L}_iu(s)\rangle_{L^2}+ \langle \Lambda^k \mathscr{L}^2_i u(s),\Lambda^k u(s)\rangle_{L^2}\right) \mathrm{d}s\\
&=-\mu \langle \Delta \Lambda^k u(s), \Lambda^k u(s) \rangle_{L^2}+\frac12\left\langle \left(\theta_r(\|\nabla u_1\|_\infty)\Lambda^k\vert\nabla u_1\vert^2 \right)-\left(\theta_r(\|\nabla u_2\|_\infty)\Lambda^k\vert\nabla u_2\vert^2 \right), \Lambda^k u\right\rangle_{L^2}\mathrm{d}s.
\end{align*}
Again using Lemma \ref{operator} and the fact that $\langle \Lambda^k \Delta u, \Lambda^k u \rangle_{L^2}= -\| \Lambda^{k+1} u\|^2_{L^2}$, we get
\begin{align}\label{ito-hk}
    &\mathrm{d}\| \Lambda^k u\|^2_{L^2}- 2\sum^n_{i=1}\langle  \Lambda^k \mathscr{L}_i u,\Lambda^k u\rangle_{L^2}\mathrm{d}W_i(s) \notag \\
    & \lesssim \left(\|u(s)\|^2_{H^k}+ \left\langle \left(\theta_r(\|\nabla u_1(s)\|_\infty)\Lambda^k\vert\nabla u_1(s,x)\vert^2 \right)-\left(\theta_r(\|\nabla u_2(s)\|_\infty)\Lambda^k\vert\nabla u_2(s,x)\vert^2 \right), \Lambda^k u\right\rangle_{L^2}\right)\mathrm{d}s
\end{align}
The nonlinearity on the RHS of \eqref{ito-hk} can be written as
 \begin{align*}
 &\left\langle \theta_r(\|\nabla u_1\|_\infty)\Lambda^k\vert\nabla u_1\vert^2 -\theta_r(\|\nabla u_2\|_\infty)\Lambda^k\vert\nabla u_2\vert^2, \Lambda^k u\right\rangle_{L^2}\\
    &=\left\langle \left(\theta_r(\|\nabla u_1\|_\infty) -\theta_r(\|\nabla u_2\|_\infty)\right)\Lambda^k\vert\nabla u_1\vert^2, \Lambda^k u\right\rangle_{L^2}+ \left\langle \theta_r(\| \nabla u_2\|_\infty)\Lambda^k (\nabla u_1 \cdot \nabla u), \Lambda^k u \right\rangle_{L^2}\\
    &\qquad+\left\langle \theta_r(\| \nabla u_2\|_\infty)\Lambda^k(\nabla u_2 \cdot \nabla u), \Lambda^k u \right\rangle_{L^2}.
\end{align*}
The first term can be estimated as 
\begin{align*}
    &\left|\left\langle \left(\theta_r(\|\nabla u_1\|_\infty) -\theta_r(\|\nabla u_2\|_\infty)\right)\Lambda^k\vert\nabla u_1\vert^2, \Lambda^k u\right\rangle_{L^2}\right|\le \|u\|_{H^k}\left| \langle \Lambda^{k-1}|\nabla u_1|^2, \Lambda^{k+1} u \rangle_{L^2}\right|\\
    &\qquad\le \|u\|_{H^k}\| \Lambda^{k-1}(\nabla u_1 \cdot \nabla u_1)\|_{L^2}\| \Lambda^{k+1} u\|^2_{L^2}\lesssim \|u\|_{H^k}\| \nabla u_1\|_\infty \|\Lambda^k u_1\|_{L^2}\|\Lambda^{k+1}u\|_{L^2}\\
    &\qquad \lesssim N^2 \|u\|_{H^k}\|\Lambda^{k+1}u\|_{L^2} \lesssim \frac{1}{\varepsilon}\|u\|^2_{H^k}+ \varepsilon \| \Lambda^{k+1} u \|^2_{L^2},
\end{align*}
where we have used Lemma \ref{kato} in the third inequality. Concerning the second term we have, 
\begin{align*}
    \left|\left\langle \theta_r(\| \nabla u_2\|_\infty)\Lambda^k (\nabla u_1 \cdot \nabla u), \Lambda^k u \right\rangle_{L^2}\right|&\le (\| \nabla u_1\|_\infty\|\Lambda^{k-1} \nabla u\|_{L^2}+\| \Lambda^{k-1}\nabla u_1\|_{L^2}\| \nabla u\|_\infty )\|\Lambda^{k+1} u \|_{L^2}\\
    &\le \| u_1\|_{H^k}\|u\|_{H^k}\| \Lambda^{k+1} u\|_{L^2} \lesssim \frac{1}{\varepsilon}\| u_1\|^2_{H^k}\|u\|^2_{H^k} + \varepsilon \|\Lambda^{k+1} u\|^2_{L^2}.
\end{align*}
Similarly, the third term gives 
\begin{align*}
    \left|\left\langle \theta_r(\| \nabla u_2\|_\infty)\Lambda^k (\nabla u_2 \cdot \nabla u), \Lambda^k u \right\rangle_{L^2}\right|&\le (\| \nabla u_2\|_\infty\|\Lambda^{k-1} \nabla u\|_{L^2}+\| \Lambda^{k-1} u_2\|_{L^2}\| \nabla u\|_\infty )\|\Lambda^{k+1} u \|_{L^2}\\
    &\le \| u_2\|_{H^k}\|u\|_{H^k}\| \Lambda^{k+1} u\|_{L^2} \lesssim \frac{1}{\varepsilon}\| u_2\|^2_{H^k}\|u\|^2_{H^k} + \varepsilon \|\Lambda^{k+1} u\|^2_{L^2}.
\end{align*}
Adding the three estimates and choosing $\varepsilon$ such that $3 \varepsilon < \mu$, we get from \eqref{ito-hk}
\begin{align}\label{est2}
    \mathrm{d}\|\Lambda^k u\|^2_{L^2} - 2\sum^n_{i=1}\langle  \Lambda^k \mathscr{L}_i u,\Lambda^k u\rangle_{L^2}\mathrm{d}W_i(s)\le (1+\|u_1\|^2_{H^k}+\|u_2\|^2_{H^k})\|u\|^2_{H^k}\mathrm{d}s.
\end{align}
Finally, we add \eqref{uniquel2} and \eqref{est2} to obtain
\begin{align}\label{est3}
    \mathrm{d}\|u\|^2_{H^k} &\lesssim \sum^n_{i=1}\left( \langle  \mathscr{L}_i u, u\rangle_{L^2}+ \langle  \Lambda^k \mathscr{L}_i u,\Lambda^k u\rangle_{L^2}\right)\mathrm{d}W(s)\notag \\
&\qquad+(1+\|u_1\|_{H^k}+\|u_1\|_{H^k}^2+\|u_2\|_{H^k}+\|u_2\|_{H^k}^2) \|u\|^2_{H^k}\mathrm{d}s.
\end{align}
Defining $X_t= -\int^t_0 (1+\|u_1\|_{H^k}+\|u_1\|_{H^k}^2+\|u_2\|_{H^k}+\|u_2\|_{H^k}^2)\mathrm{d}s $, we rewrite \eqref{est3} in Gr\"{o}nwall's form
\begin{align*}
exp(X_t)\|u(t)\|^2_{H^k} \lesssim \sum^n_{i=1}\int^t_0 exp(X_s)\left( \langle  \mathscr{L}_i u, u\rangle_{L^2}+ \langle  \Lambda^k \mathscr{L}_i u,\Lambda^k u\rangle_{L^2}\right) \mathrm{d}W_i(s).
\end{align*}
Taking expectation allows us to conclude
$\mathbb{E} \left(exp(X_t)\|u(t)\|^2_{H^k} \right) \le 0$ and this implies $\|u(t)\|^2_{H^k}=0$ a.s. as $X_t$ is finite and thus the uniqueness follows.
\subsection{Global Existence of Solution to Truncated Equation}\label{subsec:existence}
To show the well-posedness of the truncated equation \eqref{truncated}, we consider the following approximation 
\begin{align*}
\mathrm{d}u^\gamma_r(s,x)&=\left(V(x)+\mu\Delta u^\gamma_r(s,x)-\frac12\theta_r(\|\nabla u^\gamma_r(s)\|_\infty)\vert\nabla u^\gamma_r(s,x)\vert^2\right)\mathrm{d}s \\
&\qquad + \gamma \Delta^{k'} u^\gamma_r(s,x)\mathrm{d}s+\sum^n_{i=1}\mathscr{L}_i u^\gamma_r(s,x) \circ \mathrm{d}W_i(s)
\end{align*}
where $\gamma>0$ and $k' = 2[k]+1$. One can re-write it in It\^{o}'s form as follows
\begin{align}\label{regular}
\mathrm{d}u^\gamma_r(s,x)&= \left(V(x)+\mu \Delta u^\gamma_r(s,x)+ \frac12 \sum^n_{i=1}\mathscr{L}_i^2 u^\gamma_r(s,x)-\frac12\theta_r(\|\nabla u^\gamma_r\|_\infty)\vert\nabla u^\gamma_r(s,x)\vert^2\right)\mathrm{d}s \notag\\
    &\qquad + \gamma \Delta^{k'} u^\gamma_r(s,x)\mathrm{d}s+\sum^n_{i=1}\mathscr{L}_i u^\gamma_r(s,x)\mathrm{d}W_i(s)
\end{align}
We understand Eq. \eqref{regular} in the mild sense (see \cite[Chapter 6]{daprato}). Since we are required to perform higher-order computations like $\Lambda^k \Delta u$, we need the regularization by $\Delta^{k'}$ in Eq. \eqref{regular}. We now present the well-posedness result for Eq. \eqref{regular}.

\begin{prop}\label{WP-regularised}
Let $\gamma$ and $r$ be two positive numbers. For an initial data $v_0\in H^k(\mathbb{T}^n)$, there exists a unique global strong solution $u^\gamma_r$ of \eqref{regular} taking values in $L^2(\Omega; C([0,T]; H^k(\mathbb{T}^n))$. Additionally, its paths have extra regularity $C([\delta,T];H^{k+2}(\mathbb{T}^n))$, for every $0<\delta<T$.
\end{prop}
\begin{proof}
    We use the fixed point iteration argument to construct a solution, where we drop the scripts $\gamma$ and $r$ for convenience. Let $S(t)$ denote the semi-group generated by the operator $A:=\gamma \Delta^{k'}$ and write the mild formulation of \eqref{regular} as follows
    $$u(t)=(Fu)(t),$$
    where,    \begin{align}\label{mild}
        &(Fu)(t)\notag\\
        &= S(t)u_0 -\frac12\int^t_0S(t-s)\underbrace{\theta_r(\|\nabla u(s)\|_\infty)\vert\nabla u(s,x)\vert^2}_{:=F_1u} ~\mathrm{d}s \notag\\
        &~~~+ \int^t_0 S(t-s)\underbrace{\left( V(x)+\mu \Delta u(s)+\frac12\sum^n_{i=1}\mathscr{L}^2_i u(s) \right)}_{:=F_2u}~\mathrm{d}s +\sum^n_{i=1} \int^t_0 S(t-s)\mathscr{L}_i u(s)~\mathrm{d}W_i(s).
    \end{align}
    We claim that the map $F$ is a contraction on $\mathcal{W}_T:=L^2(\Omega; C([0,T];H^k(\mathbb{T}^n)))$. To prove this, first we need to make sure that $Fu \in \mathcal{W}_T$ for each $u \in \mathcal{W}_T$. Therefore let us examine each term in \eqref{mild} for $u \in \mathcal{W}_T$.
    \begin{itemize}
        \item[-] For the first term note that $S(t)$ is bounded in the space $H^k(\mathbb{T}^n)$, from the appendix we have $\| S(t) u\|_{H^k(\mathbb{T}^n)} \le C_\alpha \|(I-A)^\alpha S(t) u\|_{L^2(\mathbb{T}^n)} \le C/t^\alpha \|u\|_{L^2(\mathbb{T}^n)}$. This implies that $S(t)u_0 \in \mathcal{W}_T$.
        \item [-] For the term $F_1u$, we claim that the map $u \mapsto F_1u$ is Lipschitz continuous from $H^k(\mathbb{T}^n)$ to $L^2(\mathbb{T}^n)$. To see this we follow \cite{crisan}; let $u_1$ and $u_2$ be two elements from $H^k(\mathbb{T}^n)$. Let $B(0,r)$ denote the ball of radius r centered at zero in $H^k(\mathbb{T}^n)$ and consider the following cases:
        \begin{itemize}
            \item[(1)] When both of them are outside $B(0,r+1)$. This is trivial, as the  cut-off function vanishes and $F_1(u_1)=0=F_1(u_2)$.
            \item[(2)] When one of them is in $B(0,r+1)$, say $u_1$, and other outside $B(0,r+1)$. Then $\|\nabla u_1\|_\infty \le C\|u_1\|_{H^k(\mathbb{T}^n)} \le C(r+1)$, therefore $\|F_1 u_1\|_{L^2(\mathbb{T}^n)} \le C_r$.  Also as $F_1 u_2=0$ and $\|u_1-u_2\|_{H^k(\mathbb{T}^n)} \ge c_r$ for some $c_r>0$ we get 
            \begin{align*}
                \| F_1(u_1)-F_1(u_2)\|_{L^2(\mathbb{T}^n)} = \| F_1(u_1)\|_{L^2(\mathbb{T}^n)} \le \frac{C_r}{c_r} \| u_1- u_2\|_{H^k(\mathbb{T}^n)}.
            \end{align*}
        \item[(3)] When both $u_1,u_2$ are in $B(0,r+1)$, consider 
        \begin{align*}
            F_1(u_1)-F_1(u_2)&= \theta_r(\|\nabla u_1\|_\infty)|\nabla u_1|^2- \theta_r(\| \nabla u_2\|_\infty) | \nabla u_2|^2\\
            &= \theta_r(\|\nabla u_1\|_\infty)\nabla u_1\cdot (\nabla u_1 -\nabla u_2) + (\nabla u_1 -\nabla u_2)\cdot \nabla u_2 \theta_r(\|\nabla u_2\|_\infty)\\
            &\qquad +\nabla u_1 \cdot \nabla u_2( \theta_r(\|\nabla u_1\|_\infty) -\theta_r(\| \nabla u_2\|_\infty))
        \end{align*}
        and therefore
        \begin{align*}
            \|F_1(u_1)-F_1(u_2)\|_{L^2(\mathbb{T}^n)} &\le \theta_r(\|\nabla u_1\|_\infty)\|\nabla u_1\|_\infty\| \nabla(u_1-u_2)\|_{L^2}\\
            &\qquad+ \theta_r(\|\nabla u_2\|_\infty)\|\nabla u_2\|_\infty\| \nabla(u_1-u_2)\|_{L^2}\\
            &\qquad+\|\nabla u_1\|_\infty\|\nabla u_2\|_\infty \|\theta_r(\|\nabla u_1\|_\infty) -\theta_r(\| \nabla u_2\|_\infty)\|_{L^2}\\
            &\le C( 2r\|u_1-u_2\|_{H^k}+ (r+1)^2\| u_1 -u_2\|_{H^k}).
        \end{align*}
        where we use the fact that the map $\theta_r$ is Lipschitz.
        \end{itemize}
     Hence we get that $F_1u \in L^2( \Omega ; C([0,T]; H^k(\mathbb{T}^n)))$ and, using the first inequality from Lemma \ref{frac}, we conclude that $\int^t_0 S(t-s)F_1(u)\,ds \in \mathcal{W}_T$.
\item[-] For $u\in \mathcal{W}_T$, $\left(\mu \Delta u + \sum^n_{i=1}\frac12\mathscr{L}^2_i u\right) \in L^2(\Omega; C([0,T];L^2(\mathbb{T}^n)))$ and using the first inequality from Lemma \ref{frac}, we have $\int^t_0 S(t-s)\Delta u\mathrm{d}s \in \mathcal{W}_T$. Similarly, for $V\in L^2(\mathbb{T}^n), \int^t_0 S(t-s)V \mathrm{d}s \in \mathcal{W}_T$. Thus $\int^t_0 S(t-s)F_2u\mathrm{d}s \in \mathcal{W}_T.$
\item[-] Following \cite{crisan}, we consider the stochastic term in \eqref{mild}, we use the second inequality from Lemma \ref{frac} and the fact that $\| \nabla u\|^2_{L^2} \le \| u\|^2_{H^k(\mathbb{T}^n)}$ to conclude that $\sum^n_{i=1}\int^t_0S(t-s) \mathscr{L}_i u(s) \mathrm{d}W_i(s)\in \mathcal{W}_T$ .
\end{itemize}
This establishes that each term in $F$ is in $\mathcal{W}_T$ and therefore $F$ takes values in $\mathcal{W}_T$. Arguing as in last paragraph one can prove the Lipschitz continuity of $F$ on $\mathcal{W}_T$ and moreover $F$ can be shown to be a contraction for small $T>0$. This implies existence and uniqueness of a local solution where the time interval depends on the $L^2(\Omega; H^k(\mathbb{T}^n))$ norm of $u_0$. To extend this solution to $[0,T]$, we shall show that $u$ satisfies $\sup_{[0,T]} E \left[ \|u(t)\|_{H^k(\mathbb{T}^n)}\right] \le C(T)$. Therefore, let $u$ satisfy $u(t)=(Fu)(t)$ on $[0,T]$, and write
\begin{align*}
    & E \left[ \|u(t)\|^2_{H^k(\mathbb{T}^n)}\right] \\
    &\le 4\left(E\left[\|S(t)u_0\|^2_{H^k(\mathbb{T}^n)}\right] +E\left[\left\|\int^t_0S(t-s)\theta_r(\|\nabla u(s)\|_\infty)\vert\nabla u(s)\vert^2 ~\mathrm{d}s\right\|^2_{H^k(\mathbb{T}^n)}\right]\right.\\
    &\qquad\left.+ E\left[\left\|\int^t_0 S(t-s) \left( V(x)+\mu \Delta u(s)+\frac12\sum^n_{i=1}\mathscr{L}^2_i u(s)\right) ~\mathrm{d}s\right\|^2_{H^k(\mathbb{T}^n)}\right]\right.\\
    &\left.\qquad+ E\left[ \left\| \sum^n_{i=1}\int^t_0 S(t-s)\mathscr{L}_i u(s)~\mathrm{d}W_i(s)\right\|^2_{H^k(\mathbb{T}^n)}\right]\right)
    \end{align*}
    \begin{align*}
    &\le C\left(E\left[\|S(t)u_0\|^2_{H^k(\mathbb{T}^n)}\right] +E\left[\left\|\int^t_0S(t-s)\theta_r(\|\nabla u(s)\|_\infty)\vert\nabla u(s)\vert^2 ~\mathrm{d}s\right\|^2_{H^k(\mathbb{T}^n)}\right]\right.\\
    &\qquad\left.+ E\left[\left\|\int^t_0 S(t-s)  V(x)\right\|^2_{H^k(\mathbb{T}^n)}\right]+E\left[\left\|\int^t_0S(t-s)\mu \Delta u(s)\right\|^2_{H^k(\mathbb{T}^n)}\right]\right.\\
    &\left.+E\left[\left\|\int^t_0 S(t-s)\frac12\sum^n_{i=1}\mathscr{L}^2_i u(s) ~\mathrm{d}s\right\|^2_{H^k(\mathbb{T}^n)}\right]+ E\left[ \left\| \sum^n_{i=1}\int^t_0 S(t-s)\mathscr{L}_i u(s)~\mathrm{d}W_i(s)\right\|^2_{H^k(\mathbb{T}^n)}\right]\right)  \\
        & \le C \left( E \| u_0\|^2_{L^2(\mathbb{T}^n)}+ E\left[ \int^t_0\frac{1}{(t-s)^{\frac{k}{2k'}}}\|\theta_r(\|\nabla u(s)\|_\infty)|\nabla u(s)|^2\|^2_{L^2(\mathbb{T}^n)}\mathrm{d}s\right] \right.\\
        &\qquad\left.+  T^{2-s/k'}E\left[\sup_{[0,T]}\| V(x)\|^2_{L^2(\mathbb{T}^n)}\right]+\mu E\left[ \int^t_0 \frac{1}{(t-s)^{\frac{k}{2k'}}}\|\Delta u(s)\|^2_{L^2(\mathbb{T}^n)}\mathrm{d}s\right] \right.\\
        &\qquad\left.+E\left[\left\|\int^t_0 \sum^n_{i=1}\mathscr{L}^2_i u(s) ~\mathrm{d}s\right\|^2_{H^k(\mathbb{T}^n)}\right]+ E\left[\left\|\int^t_0 \sum^n_{i=1}\mathscr{L}_i u(s) \mathrm{d}W_i(s)\right\|^2_{H^k(\mathbb{T}^n)}\right] \right)\\
        &\le C \left( E \| u_0\|^2_{L^2(\mathbb{T}^n)}+ T^{2-k/k'}\| V(x)\|^2_{L^2(\mathbb{T}^n)}+\left(r+\mu+C(a,b)\right)E\left[ \int^t_0 \frac{1}{(t-s)^{\frac{k}{2k'}}}\|u(s)\|^2_{H^k(\mathbb{T}^n)}\mathrm{d}s \right]\right),
\end{align*}
where to bound the $H^k$ norms we use the inequalities available for fractional powers mentioned in the appendix. Applying Gr\"onwall's inequality we get the desired \textit{a priori} bound
\begin{align*}
    \sup_{[0,T]}E \left[ \|u(t)\|^2_{H^k(\mathbb{T}^n)}\right] \le C.
\end{align*}
We shall now show that the solution $u$ in fact gains higher regularity. Notice that $F_1u$ and $F_2u$ are a.s. in $C([0,T];L^2(\mathbb{T}^n))$ as shown above, thus using the first inequality from Lemma \ref{frac}, we observe that the two Lebesgue integrals in Eq. \eqref{mild} belong to the space $C([0,T]; H^{k+2}(\mathbb{T}^n))$. For the treatment of the stochastic integral in Eq. \eqref{mild}, we make use of the second inequality in Lemma \ref{frac} to conclude that it belongs to $C([0,T];H^{k+2}(\mathbb{T}^n))$ a.s. . Finally, for the first term $S(t)u_0$ in Eq. \eqref{mild}, one can follow the same ideas as in \cite[Lemma 19]{crisan} and make use of properties of the semigroup $S(t)$ to conclude that $S(t)u_0$ belongs to $C([\delta,T]; H^{k+2}(\mathbb{T}^n))$, for every $0 < \delta <T$. Thus we conclude that $u=Fu \in  L^2(\Omega;C([\delta,T]; H^{k+2}(\mathbb{T}^n)))$, for every $0 < \delta <T$.
\end{proof}

\subsubsection{Passing from Regularised to Truncated Equation }\label{compactness}
We now aim to show that the family of solutions $\{u^\gamma_r\}_{\gamma >0}$ to \eqref{regular} is compact in a suitable sense and one can extract a converging sub-sequence that converges to a solution to \eqref{truncated} . We start using It\^{o}'s formula to derive some \textit{a priori} estimates.
In the stochastic setting, because of the presence of the additional variable $\omega \in \Omega$, strong compactness is not expected and one traditionally treats this issue by extracting weak compactness of the laws of the approximate solutions. This is done via tightness and Prokhorov's theorem. We keep the details for later and start by establishing the energy estimates for the approximate solutions $\{u^\gamma_r\}_{\gamma >0}$. 

% For this we propose to show that the corresponding laws $\{\rho^\gamma_r\}_{\gamma>0}$ of $\{u^\gamma_r\}_{\gamma>0}$ are tight in the Polish space $E:=C([0,T]; H^k(\mathbb{T}^n))$. \\
% Assuming that the sequence $\{\rho^\gamma_r\}$ is tight in E, using Prokhorov we  can extract a weakly convergent subsequence of measures $\{\rho^\gamma_r\}$ (still denoting by same).
\begin{prop}\label{apriori}
The following \textit{a priori} estimates hold 
\begin{align}\label{fourth}
\begin{aligned}
   & E \left[ \sup_{t \in [0,T]} \| u^\gamma_r(t) \|^4_{H^k(\mathbb{T}^n)} \right]\le C_1,\\
    &E \left[ \int^T_0 \int^T_0 \frac{\| u^\gamma_r(t)- u^\gamma_r(s)\|^4_{H^{-M}(\mathbb{T}^n)}}{|t-s|^{1+4\alpha}}\mathrm{d}t\mathrm{d}s\right] \le C_2,
    \end{aligned}
\end{align}
for some $\alpha>0, M \in \mathbb{N}$,
uniformly in $\gamma$. 
\end{prop}
\begin{proof}
Using \eqref{regular} we first write the evolution of the $L^2$ norm of $u^\gamma_r$
\begin{align*}
        \frac12 \|u^\gamma_r(t)\|^2_{L^2}&= \frac12 \|u^\gamma_r(0)\|^2_{L^2}\\
        &\quad+\int^t_0\left\langle V(x)+\mu\Delta u^\gamma_r(s) -\frac12 \theta_r(\|\nabla u^\gamma_r(s)\|_{L^\infty})| \nabla u^\gamma_r(s)|^2, u^\gamma_r(s) \right\rangle_{L^2}\,ds\\
    &\quad +\sum^n_{i=1}\int^t_0 \langle \mathscr{L}_i u^\gamma_r(s),u^\gamma_r(s) \rangle_{L^2} \,dW_i(s)+\int^t_0 \gamma \langle \Delta^{k'} u^\gamma_r(s), u^\gamma_r(s) \rangle_{L^2}\,ds\\
&\quad+\frac12\sum^n_{i=1}\int^t_0\left(\langle \mathscr{L}^2_i u^\gamma_r(s),u^\gamma_r(s) \rangle_{L^2} + \langle \mathscr{L}_i u^\gamma_r(s), \mathscr{L}_i u^\gamma_r(s) \rangle_{L^2}\right)\,ds\\
    &\lesssim \frac12 \|u^\gamma_r(0)\|^2_{L^2}+\int^t_0\|V\|_{L^2} \|u^\gamma_r(s)\|_{L^2}\,ds-\mu \int^t_0 \| \nabla u^\gamma_r(s)\|^2_{L^2}\,ds+\frac12\int^t_0 (r+1)^2\|u^\gamma_r(s)\|_{L^2}\,ds\\
    &\quad+\sum^n_{i=1}\int^t_0 \langle \mathscr{L}_i u^\gamma_r(s),u^\gamma_r(s) \rangle_{L^2} \,dW_i(s)-\gamma \int_0^t \| \Delta^{-k'/2} u^\gamma_r(s)\|^2_{L^2}\,ds+\int^t_0 \|u^\gamma_r(s)\|^2_{L^2}\mathrm{d}s,
\end{align*}
where we used Lemma \ref{operator} to estimate the last two terms in the equality. Moreover, using $ab \le \frac{a^2+b^2}{2}$ for $a,b \in \mathbb{R}$ and dropping the negative terms, we can write
\begin{align}\label{Ltwo}
\|u^\gamma_r(t)\|^2_{L^2}
\lesssim \|u^\gamma_r(0)\|^2_{L^2}+ T\left(\|V\|^2_{L^2}+ (r+1)^4/2\right) + \int^t_0 \| u^\gamma_r(s)\|^2_{L^2}\,ds+\sum^n_{i=1}\int^t_0 \langle \mathscr{L}_i u^\gamma_r(s),u^\gamma_r(s) \rangle_{L^2} \,dW_i(s).
\end{align}
Next, using the regularity from Proposition \ref{WP-regularised}, we can write the following identity using \eqref{regular}  
\begin{align*}
\mathrm{d}\Lambda^ku^\gamma_r(s,x)&= \sum^n_{i=1}\Lambda^k\mathscr{L}_i\nabla u^\gamma_r(s,x) \mathrm{d}W(s)+ \gamma \Lambda^k\Delta^{k'} u^\gamma_r\mathrm{d}s\\
&+\left(\Lambda^kV(x)+\mu\Lambda^k\Delta u^\gamma_r(s,x)+\frac12 \sum^n_{i=1}\Lambda^k\mathscr{L}_i^2 u^\gamma_r(s,x)-\frac12\theta_r(\|\nabla u^\gamma_r\|_\infty)\Lambda^k\vert\nabla u^\gamma_r(s,x)\vert^2\right)\mathrm{d}s,
    \end{align*}
therefore the evolution of $\|\cdot\|_{H^k}$ of $u$ is as follows 
\begin{align}\label{fract}
\begin{aligned}
\frac{1}{2}\left\|\Lambda^{k} u_{r}^{\gamma}(t)\right\|_{L^{2}}^{2}&=\frac{1}{2}\left\|\Lambda^{k} u_{r}^{\gamma}(0)\right\|_{L^{2}}^{2}-\frac12\int_{0}^{t}\left\langle\theta_{r}\left(\left\|\nabla u_{r}^{\gamma}(s)\right\|_{L^{\infty}}\right) \Lambda^{k}\left(|\nabla u_{r}^{\gamma}(s)|^2\right), \Lambda^{k} u_{r}^{\gamma}(s)\right\rangle_{L^{2}} \mathrm{d}s \\
&+\sum^n_{i=1}\int_{0}^{t}\left\langle\Lambda^{k} \mathscr{L}_i u_{r}^{\gamma}(s), \Lambda^k u_{r}^{\gamma}(s)\right\rangle_{L^{2}} \mathrm{~d} W_i(s)+\int_{0}^{t}\left\langle\gamma \Lambda^k \Delta^{k'} u_{r}^{\gamma}(s), \Lambda^k u_{r}^{\gamma}(s)\right\rangle_{L^{2}} \mathrm{d}s \\
&+\mu \int_{0}^{t}\left\langle\Lambda^k \Delta u_{r}^{\gamma}(s), \Lambda^k u_{r}^{\gamma}(s)\right\rangle_{L^{2}} \mathrm{d}s+\frac{1}{2} \int_{0}^{t}\left\langle\Lambda^k V(x), \Lambda^k  u_{r}^{\gamma}(s)\right\rangle_{L^{2}} \mathrm{d} s\\
& +\frac12 \sum^n_{i=1}\int^t_0\langle\Lambda^k \mathscr{L}^2_i u^\gamma_r(s), \Lambda^k u^\gamma_r(s) \rangle_{L^2}\mathrm{d}s +\frac12 \sum^n_{i=1}\int^t_0 \langle \Lambda^k \mathscr{L}_i u^\gamma_r(s),\Lambda^k \mathscr{L}_i u^\gamma_r(s) \rangle \mathrm{d}s.
\end{aligned}
\end{align}
To estimate the nonlinear term we notice that 
\begin{align*}
\left| \left\langle \Lambda^{k}\left(|\nabla u_{r}^{\gamma}|^2\right), \Lambda^{k} u_{r}^{\gamma}\right\rangle_{L^{2}}\right| &\le \| \Lambda^{k}\left(|\nabla u_{r}^{\gamma}|^2\right)\|_{L^2} \|\Lambda^{k} u_{r}^{\gamma}\|_{L^2}\\
&\lesssim \|\Lambda^{k}\left(\nabla u_{r}^{\gamma}\right)\|_{L^2}\|\nabla u_{r}^{\gamma}\|_{L^\infty}\|\Lambda^{k} u_{r}^{\gamma}\|_{L^2}\\
&\lesssim  \|\Lambda^{k+1} u_{r}^{\gamma}\|_{L^2}\|\nabla u_{r}^{\gamma}\|_{L^\infty}\|\Lambda^{k} u_{r}^{\gamma}\|_{L^2}\\
&\lesssim \frac{1}{4 \gamma} \| \nabla u^\gamma_r\|^2_{L^\infty} \|\Lambda^{k} u_{r}^{\gamma}\|^2_{L^2}+\gamma \|\Lambda^{k+1}u_{r}^{\gamma}\|^2_{L^2},
\end{align*}
where we have used Lemma \ref{kato} in the second inequality. Using this estimate in \eqref{fract}, using Lemma \ref{operator} and observing that
\begin{align*}
&\gamma \|\Lambda^{k+1}u_{r}^{\gamma}\|^2_{L^2}+ \left\langle\gamma \Lambda^k \Delta^{k'} u_{r}^{\gamma}, \Lambda^k u_{r}^{\gamma}\right\rangle_{L^{2}}=\gamma \left(\|\Lambda^{k+1}u_{r}^{\gamma}\|^2_{L^2}- \| \Lambda^{k+2[k]+1} u^\gamma_r\|^2_{L^2} \right) \le 0,\\
& \mu \left\langle \Lambda^k \Delta u^\gamma_r, \Lambda^k u^\gamma_r \right \rangle_{L^2} = -\mu \|\Lambda^k \nabla u^\gamma_r\|^2_{L^2}\le 0,
\end{align*} 
we get 
\begin{align*}
    \|\Lambda^k u^\gamma_r(t)\|^2_{L^2} &\lesssim
    \| \Lambda^k u^\gamma_r(0) \|^2_{L^2}+ \frac{1}{4\gamma} \int^t_0 \theta_r(\| \nabla u^\gamma_r(s)\|_{L^\infty}) \| \nabla u^\gamma_r(s)\|^2_{L^\infty}\| \Lambda^k u^\gamma_r(s)\|^2_{L^2}\mathrm{d}s\\
&+\sum^n_{i=1}\int_{0}^{t}\left\langle\Lambda^{k} \mathscr{L}_i u_{r}^{\gamma}(s), \Lambda^k u_{r}^{\gamma}(s)\right\rangle_{L^{2}} \mathrm{~d} W_i(s)+ \frac12 \| \Lambda^k V\|^2_{L^2}T+ \int^t_0 \| \Lambda^k u^\gamma_r(s)\|^2_{L^2}\mathrm{d}s.
\end{align*}
Therefore, we have
\begin{align}\label{semi-norm}
\begin{aligned}
    \|\Lambda^k u^\gamma_r(t)\|^2_{L^2}&\lesssim \| \Lambda^k u^\gamma_r(0)\|^2+\frac12 \| \Lambda^k V\|^2_{L^2}T +C(r,\gamma) \int^t_0 \| \Lambda^k u^\gamma_r(s)\|^2_{L^2}\mathrm{d}s\\ &\quad+\sum^n_{i=1}\int_{0}^{t}\left\langle\Lambda^{k} \mathscr{L}_i u_{r}^{\gamma}(s), \Lambda^k u_{r}^{\gamma}(s)\right\rangle_{L^{2}} \mathrm{~d} W_i(s).
    \end{aligned}
\end{align}
Combining \eqref{Ltwo} and \eqref{semi-norm} we get,
\begin{align}\label{Hk-estimate}
\begin{aligned}
 \| u^\gamma_r\|^2_{H^k} &\le \|u^\gamma_r(0)\|^2_{H^k}+C(V,T,r)+ C(r,\gamma)  \int^t_0 \|u^\gamma_r(s)\|^2_{H^k}\mathrm{d}s\\ &+\sum^n_{i=1}\int_{0}^{t}\left\langle \mathscr{L}_i u_{r}^{\gamma}(s),  u_{r}^{\gamma}(s)\right\rangle_{L^{2}} \mathrm{~d} W_i(s)+\sum^n_{i=1}\int_{0}^{t}\left\langle\Lambda^{k} \mathscr{L}_i u_{r}^{\gamma}(s), \Lambda^k u_{r}^{\gamma}(s)\right\rangle_{L^{2}} \mathrm{~d} W_i(s)
\end{aligned}
\end{align}
% Following \cite{alonso,crisan}, one denotes 
% \begin{align*}
%     M_t = 2\sqrt{\nu} \int^t_0\left(\left\langle \nabla u^\gamma_r, u^\gamma_r \right\rangle_{L^2}+ \left\langle \Lambda^k \nabla u^\gamma_r, \Lambda^k u^\gamma_r \right\rangle_{L^2} \right)\mathrm{~d}W(s),
% \end{align*}
Denoting by $M_t:=\sum^n_{i=1}\int_{0}^{t}\left\langle \mathscr{L}_i u_{r}^{\gamma}(s),  u_{r}^{\gamma}(s)\right\rangle_{L^{2}} \mathrm{~d} W_i(s)+\sum^n_{i=1}\int_{0}^{t}\left\langle\Lambda^{k} \mathscr{L}_i u_{r}^{\gamma}, \Lambda^k u_{r}^{\gamma}\right\rangle_{L^{2}} \mathrm{~d} W_i(s)$, we  apply Gr\"onwall's inequality and square both sides of \eqref{Hk-estimate} to obtain
\begin{align}\label{squaring-for-BDG}
\|u^\gamma_r(t)\|^4_{H^k}\lesssim e^{Ct}(C^2(V,T,r)+\|u^\gamma_r(0)\|^4_{H^k}+|M_t|^2)
\end{align}
Taking supremum over time and applying expectation gives
\begin{align*}
      E \left[\sup_{s \in [0,T]}\|u^\gamma_r(s)\|^4_{H^k} \right]\le e^{CT}\left(\|u^\gamma_r(0)\|^4_{H^k}+C^2(V,T,r)+E\left[\sup_{s \in [0,T]}|M_s|^2\right] \right).
 \end{align*}
 We now proceed exactly as after Eqn (4.14) in \cite[Prop 4.5]{alonso} to conclude 
 \begin{align*}
     E\left[\sup _{t \in[0, T]}\|u^\gamma_r(t)\|_{H^k(\mathbb{T}^n)}^4\right] \leq C(T) .
 \end{align*}
% using the BDG inequality, we obtain 
% \begin{align*}
% E \left[ \sup_{s \in [0,T]} |M_t|^2\right] \le C E [[M_t]] &=C\nu \int^t_0\left(\left\langle \nabla u^\gamma_r, u^\gamma_r \right\rangle_{L^2}+ \left\langle \Lambda^k \nabla u^\gamma_r, \Lambda^k u^\gamma_r \right\rangle_{L^2} \right)^2\mathrm{~d}s\\
% &\le C \nu\int^t_0 \|u\|^4_{H^k}\,ds,
% \end{align*}
% where we use inequality $ | \langle \nabla u,u\rangle_{L^2} + \langle \Lambda^k \nabla u, \Lambda^k u \rangle_{L^2}| \lesssim \| u\|^2_{H^k}$ proved in \cite[proposition 4.5]{alonso}. This implies
% \begin{align*}
%      E \left[\sup_{s \in [0,T]}\|u^\gamma_r\|^4_{H^k} \right]\lesssim  e^{2\widetilde{C}t}\left(\|u^\gamma_r(0)\|^4_{H^k}+C^2(V,T,r)+ C \int^t_0 \|u\|^4_{H^k}\,ds\right).
% \end{align*}
% An application of Gr\"onwall's inequality gives
% \begin{align*}
%     E \left[ \sup_{s\in [0,T]}\|u^\gamma_r(s)\|^4_{H^k}\right] \le C(\nu,V,T,r).
% \end{align*}
For the second part of the proposition we follow the techniques in \cite{crisan}. Using Eq.  \eqref{regular}, we consider the difference, for $0\le s<t\le T$,
\begin{align*}
    u^\gamma_r(t)-u^\gamma_r(s)&= \sum^n_{i=1}\int^t_s \mathscr{L}_i u^\gamma_r(\tau) \mathrm{d}W_i(\tau)+ \gamma \int^t_s \Delta^{k'} u^\gamma_r(\tau)\mathrm{d}\tau\\
    &+\int^t_s\left(V(x)+\mu\Delta u^\gamma_r(\tau)+\frac12\sum^n_{i=1}\mathscr{L}^2_i u^\gamma_r(\tau)-\frac12\theta_r(\|\nabla u^\gamma_r(\tau)\|_\infty)\vert\nabla u^\gamma_r(\tau)\vert^2\right)\mathrm{d}\tau.
\end{align*}
Taking $\|\cdot\|^2_{H^{-M}(\mathbb{T}^n)}$ ($M$ to be chosen later) in both sides, applying triangle and Jensen's inequalities we get, after taking expectation,  
\begin{align}\label{fourth-estimate}
\begin{aligned}
    E \left[ \|u^\gamma_r(t)-u^\gamma_r(s) \|^4_{H^{-M}}\right] &\lesssim E\left[\left\|\sum^n_{i=1}\int^t_s\mathscr{L}_i u^\gamma_r(\tau) \mathrm{d}W_i(\tau)\right\|^4_{H^{-M}}\right]\\
    &+(t-s)^3 \int^t_s E \left[\|\gamma  \Delta^{k'} u^\gamma_r(\tau)\|^4_{H^{-M}} \right]\mathrm{d}\tau +(t-s)\|V\|^4_{H^{-M}}\\
&+(t-s)^3 \int^t_s E \left[\mu\left\|\Delta u^\gamma_r(\tau)\right\|^4_{H^{-M}}\right]\mathrm{d}\tau \\
&+(t-s)^3\int^t_s E \left[\left\|\frac12\sum^n_{i=1}\mathscr{L}^2_i u^\gamma_r(\tau)\right\|^4_{H^{-M}}\right]\mathrm{d}\tau\\
&+(t-s)^3 \int^t_s E\left[ \theta_r(\| \nabla u^\gamma_r(\tau)\|_\infty) \|| \nabla u^\gamma_r(\tau)|^2\|^4_{H^{-M}} \right] \mathrm{d}\tau.
\end{aligned}
\end{align}
Notice that $\|| \nabla u^\gamma_r|^2\|^4_{H^{-M}} \le \|| \nabla u^\gamma_r|^2\|^4_{L^2} \le \| \nabla u^\gamma_r\|^4_{L^\infty}\| u^\gamma_r \|^4_{H^k}$ and therefore the last term of \eqref{fourth-estimate} is estimated as,
\begin{align*}
\theta_r(\| \nabla u^\gamma_r(\tau)\|_\infty) \|| \nabla u^\gamma_r(\tau)|^2\|^4_{H^{-M}} \le C(r) \| u^\gamma_r\|^4_{H^k},
\end{align*}
which gives $\int^t_s E\left[ \theta_r(\| \nabla u^\gamma_r(\tau)\|_\infty) \|| \nabla u^\gamma_r(\tau)|^2\|^2_{H^{-M}} \right] \mathrm{d}\tau \le C(t-s)$, where we use \eqref{fourth}.\\
For the second term in RHS of  \eqref{fourth-estimate} we can use the fact $\| \Delta^{k'} u^\gamma_r\|_{H^{-M}}\le \|u^\gamma_r\|_{H^k}$ for $M=3[k']+2$ to conclude that
$$\int^t_sE \left[ \| \gamma \Delta^{k'} u^\gamma_r(\tau)\|^4_{H^{-M}}\right]\mathrm{d}\tau \lesssim \int^t_s E [ \| u^\gamma_r(\tau)\|^4_{H^k}]\mathrm{d}\tau\le C(T).$$
With similar reasoning the fourth term of \eqref{fourth-estimate} is also bounded above by $t-s$. For the It\^{o} correction term, we have 
\begin{align*}
    \int^t_s E \left[\left\|\frac12\sum^n_{i=1}\mathscr{L}^2_i u^\gamma_r(\tau)\right\|^4_{H^{-M}}\right]\mathrm{d}\tau &\le  \int^t_s E \left[\left\|\frac12\sum^n_{i=1}\mathscr{L}^2_i u^\gamma_r(\tau)\right\|^4_{L^2}\right]\mathrm{d}\tau\\
    &\le C(a,b) \int^t_s E\left[\|\Delta u^\gamma_r(\tau)\|^4_{L^2}+ \|\nabla u^\gamma_r(\tau)\|^4_{L^2}+\| u^\gamma_r(\tau)\|^4_{L^2} \right]\mathrm{d}\tau\\
    &\lesssim \int^t_s E \left[ \|u^\gamma_r(\tau)\|^4_{H^k}\right]\mathrm{d}\tau \le C,
\end{align*}
where $C(a,b)$ consists of $W^{2,\infty}$ norm of $a$ and $W^{1,\infty}$ norm of $b$.
The stochastic term is estimated as follows 
\begin{align*}
&E\left[\left\|\sum^n_{i=1}\int^t_s \mathscr{L}_i u^\gamma_r(\tau) \mathrm{d}W_i(\tau)\right\|^4_{H^{-M}}\right] \le E\left[\left\|\sum^n_{i=1}\int^t_s\mathscr{L}_i u^\gamma_r(\tau) \mathrm{d}W_i(\tau)\right\|^4_{L^2}\right]\\
&\qquad\le C E \left[ \left( \int^t_s \sum^n_{i=1}\|\mathscr{L}_i u^\gamma_r(\tau)\|^2_{L^2}\mathrm{d}\tau\right)^2\right]\le C(a,b)(t-s) \int^t_s E\left[\|\nabla u^\gamma_r(\tau)\|^4_{L^2}+\| u^\gamma_r(\tau)\|^4_{L^2} \right]\mathrm{d}\tau \\
&\qquad\lesssim (t-s)\int^t_s E \left[ \|u^\gamma_r(\tau)\|^4_{H^k}\right]\mathrm{d}\tau \le C(t-s)^2,
\end{align*}
where the constant $C(a,b)$ consists of $W^{1,\infty}$ norm of $a$ and $L^\infty$ norm of $b$. Putting back all the estimates in \eqref{fourth-estimate} we get
$$ E \left[ \|u^\gamma_r(t)-u^\gamma_r(s) \|^4_{H^{-M}}\right] \le C(t-s)^2, $$
and therefore, for $0<\alpha<1/2$,
\begin{align*}
    E \left[ \int^T_0 \int^T_0 \frac{\| u^\gamma_r(t)- u^\gamma_r(s)\|^4_{H^{-M}}}{|t-s|^{1+4\alpha}}\mathrm{d}t\mathrm{d}s\right] \le E \left[ \int^T_0 \int^T_0 \frac{C(T)}{|t-s|^{4\alpha-1}}\mathrm{d}t\mathrm{d}s\right]\le C(T). 
\end{align*}
\end{proof}
Having established the energy estimates for the sequence of solutions $\{u^\gamma_r\}$ to \eqref{regular}, we now state the tightness result for $\{u^\gamma_r\}$, which relies on the variant of classical Aubin-Lions Lemma (see \cite{simon}).
\begin{corollary}
    The family of laws $\{ \rho^\gamma_r\}_{\gamma >0}$ of the sequence $\{u^\gamma_r\}_{\gamma >0}$ is tight in the Polish space $E:=C([0,T];H^\beta(\mathbb{T}^n))$ where $2\le \beta <k$.
\end{corollary}
\begin{proof}
    Notice that, for $\beta <k$, we have the compact embedding of $H^k(\mathbb{T}^n)$ in $H^\beta(\mathbb{T}^n)$ and thus, applying the Corollary 9 from \cite{simon}, we get the compact embedding of $E_0:=L^\infty([0,T]; H^k(\mathbb{T}^n)) \cap \mathscr{W}^{\alpha,4}([0,T];H^{-M}(\mathbb{T}^n))$ into $C([0,T]; H^\beta(\mathbb{T}^n))$. From Proposition \ref{apriori}, we know that the family of laws $\{\rho^\gamma_r\}_{\gamma>0}$ are supported on $E_0$, therefore it suffices to prove the tightness of $\{\rho^\gamma_r\}_{\gamma>0}$ in $E_0$. 
    Define the following sets
    \begin{align*}
       & K_1:= \{ f:[0,T] \times \mathbb{T}^n: \sup_{t \in[0, T]}\|f(t)\|_{H^k}^2 \leq R_1\},\\
       & K_2:=\{f:[0,T] \times \mathbb{T}^n: \int^T_0\| f(t)\|^4_{H^{-M}}\mathrm{d}t \le R_2\},\\
       &K_3:=\left\{f:[0,T] \times \mathbb{T}^n: \int_0^T \int_0^T \frac{\|f(t)-f(s)\|_{H^{-M}}^4}{|t-s|^{1+4 \alpha}} \mathrm{d} t \mathrm{~d} s \leq R_3\right\},
    \end{align*}
    for some positive reals $R_1,R_2,R_3$. Let $\varepsilon>0$ and define $K:= K_1 \cap K_2 \cap K_3$, then we shall show that $K\in E_0$ such that
       $\rho^\gamma_r(K^c) \le \varepsilon$. By Chebyshev's inequality we have 
\begin{align*}
\rho^\gamma_r(K_1^c)=\mathbb{P}\left( u^\gamma_r(t) \in K_1^c \right)=\mathbb{P}\left( \sup_{[0,T]}\|u^\gamma_r(t)\|^2_{H^k}> R_1\right) \le \frac{1}{R_1} E \left[\sup_{[0,T]}\|u^\gamma_r(t)\|^2_{H^k} \right] \le \frac{C_1}{R_1} \le \frac{\varepsilon}{3} , 
\end{align*}
where we use Proposition \ref{apriori} and choose $R_1$ large enough.
Similarly,
\begin{align*}
\rho^\gamma_r(K_2^c)=    \mathbb{P}\left( \int^T_0\|u^\gamma_r(t)\|^4_{H^{-M}}> R_2\right)&\le \mathbb{P}\left( T \sup_{[0,T]}\|u^\gamma_r(t)\|^4_{H^{-M}}> R_2\right)\\
&\le \mathbb{P}\left( T \sup_{[0,T]}\|u^\gamma_r(t)\|^4_{H^k}> R_2\right)\\
    &\le \frac{T}{R_2} E \left[\sup_{[0,T]}\|u^\gamma_r(t)\|^4_{H^k} \right] \le \frac{TC_2}{R_2} \le \frac{\varepsilon}{3},
\end{align*}
for $R_2$ large enough and finally
\begin{align*}
\rho^\gamma_r(K_3^c)=\mathbb{P}\left( \int^T_0 \int^T_0 \frac{\|u^\gamma_r(t)-u^\gamma_r(s)\|^4_{H^{-M}}}{|t-s|^{1+4 \alpha}} \mathrm{d}t \mathrm{~d}s > R_3\right) \le \frac{\varepsilon}{3},
\end{align*}
for $R_3$ large enough. This implies $\rho^\gamma_r(K^c)\le \sum^3_{i=1}\rho^\gamma_r(K_i^c)\le \varepsilon$ and we get the desired result.
    \end{proof}

Having proved the tightness of $\{\rho^\gamma_r\}$ in $E$, one can apply Prokhorov's compactness theorem to extract a weakly convergent subsequence $\{\rho^{\gamma}_r\}$ (still denoted the same), converging to $\rho_r$ as $\gamma \to 0$.  Skorokhod's representation theorem then supplies new random variables $\{\widetilde{u}^{\gamma}_r\}_{\gamma>0}$ and $\widetilde{u}_r $ on a new probability space $(\widetilde{\Omega},\widetilde{\mathcal{F}},\widetilde{\mathbb{P}})$ with $\{\widetilde{u}^{\gamma}_r\}_{\gamma>0}$ converging to $\widetilde{u}_r$ as $\gamma$ tends to zero $\widetilde{\mathbb{P}}$ a.s., such that the laws of $\{\widetilde{u}^{\gamma}_r\}_{\gamma>0}$ and $\widetilde{u}_r$ are given by $\{\rho^\gamma_r\}_{\gamma>0}$ and $\{\rho_r\}$ respectively. This new sequence $\{\widetilde{u}^{\gamma}_r\}_{\gamma>0}$ satisfies \eqref{regular} weakly and with the convergence result in $\widetilde{\mathbb{P}}$ we get existence of solution for \eqref{truncated} in the new probabilty space $\widetilde{\mathbb{P}}$. Finally, to obtain the convergence in the original probability space, we use the Gy\"ongy-Krylov's lemma, which needs a diagonal assumption (see Lemma \ref{gyongy}) to be applied. This latter assumption is satisfied with the  global uniqueness proved in Subsection \ref{subsec:uniqueness}. Our points in this paragraph are based on standard and classical stochastic arguments and, as our estimates are similar to Crisan \textit{et al.} \cite{crisan}, we refer the reader to this article for a more comprehensive explanation.

We now give details on the  convergence of the solution of Eq. \eqref{regular} to the solution of Eq. \eqref{HJB} $\widetilde{\mathbb{P}}$ a.s. . For convenience we choose to drop the ``tilde" notation and work with $u^\gamma_r$ instead. We begin by integrating \eqref{regular} against test functions and show that
\begin{align*}
     \langle u^\gamma_r(t), \varphi \rangle_{L^2}&= \langle u^\gamma_r(0), \varphi \rangle_{L^2} +\sum^n_{i=1} \int^t_0 \left\langle u^\gamma_r(s,x), \mathscr{L}^*_i \varphi \right\rangle_{L^2}\mathrm{d}W_i(s) +\mu\int^t_0 \langle u^\gamma_r(s), \Delta \varphi \rangle_{L^2} \mathrm{d}s\\
     &+\int^t_0\langle V(x), \varphi \rangle_{L^2}\,ds+\sum^n_{i=1} \int^t_0 \left\langle u^\gamma_r(s,x), (\mathscr{L}^*_i)^2 \varphi \right\rangle_{L^2}\mathrm{d}s\\
     &-\frac12\int^t_0 \langle \theta_r(\|\nabla u^\gamma_r(s)\|_\infty)\vert\nabla u^\gamma_r(s,x)\vert^2 , \varphi(x)\rangle_{L^2}\mathrm{d}s + \int^t_0 \gamma \langle u^\gamma_r(s), \Delta^{k'} \varphi \rangle_{L^2}\mathrm{d}s
     \end{align*}
converges to 
     \begin{align*}
     \langle u_r(t), \varphi \rangle_{L^2}&= \langle u_r(0), \varphi \rangle_{L^2} +\sum^n_{i=1} \int^t_0 \left\langle u_r(s,x), \mathscr{L}^*_i \varphi \right\rangle_{L^2}\mathrm{d}W_i(s)
     +\mu \int^t_0 \langle u_r(s), \Delta \varphi \rangle_{L^2}\mathrm{d}s\\
     &+\int^t_0\langle V(x), \varphi \rangle_{L^2}\mathrm{d}s+\frac12\sum^n_{i=1} \int^t_0 \left\langle u_r(s,x), (\mathscr{L}^*_i)^2 \varphi(x) \right\rangle_{L^2}\mathrm{d}s\\
     &-\frac12\int^t_0 \langle \theta_r(\|\nabla u_r(s)\|_\infty)\vert\nabla u_r(s,x)\vert^2 , \varphi\rangle_{L^2}\mathrm{d}s,
\end{align*}
as $\gamma$ tends to zero for every $\varphi \in C^\infty(\mathbb{T}^n)$ and $t \in [0,T]$. Let us analyze each term separately as follows
\begin{itemize}
    \item[-] For the term on the LHS, notice that $| \langle u^\gamma_r(t)- u_r(t), \varphi \rangle \le \|u^\gamma_r(t)-u_r(t)\|_{L^2}\|\varphi\|_{L^2} \to 0$ as $\gamma$ tends to zero and $u^\gamma_r \to u_r$ in $C([0,T];H^\beta(\mathbb{T}^n))- \widetilde{\mathbb{P}}$ a.s. .
    \item[-] For the martingale term one can use the fact that $u^\gamma_r \to u_r$ in $C([0,T];H^\beta(\mathbb{T}^n))- \widetilde{\mathbb{P}}$ a.s. as $\gamma$ tends to zero and apply stochastic dominated convergence to the process $\langle u^\gamma_r(t)-u_r(t), \mathscr{L}^*_i\varphi \rangle_{L^2}$. 
    \item[-] For the viscosity term $\langle u^\gamma_r, \Delta \varphi \rangle_{L^2}$, we use 
    \begin{align*}
        \left|\int^t_0 \langle u^\gamma_r(s)- u_r(s), \Delta \varphi \rangle_{L^2} \mathrm{d}s \right|\le T \sup_{[0,T]}\|u^\gamma_r(s)-u_r(s)\|_{L^2}\| \Delta \varphi\|_{L^2} \to 0
    \end{align*}
    as $\gamma$ tends to zero, again using the pathwise convergence $\{u^\gamma_r\}$ to $u_r$. In a similar way the term $\left\langle u^\gamma_r(s,x), (\mathscr{L}^*_i)^2 \varphi(x) \right\rangle_{L^2}$ can be treated.
    \item[-] Let us now consider the higher order term;  noticing that \\ $\int^t_0\| u^\gamma_r(s)-u_r(s)\|_{L^2}\,ds \le T \sup_{[0,T]}\|u^\gamma_r(s)-u_r(s)\|_{L^2} \to 0$, as $\gamma$ tends to zero, we get $\int^t_0 \|u^\gamma_r(s)\|_{L^2}\mathrm{d}s \to \int^t_0\|u_r(s)\|_{L^2}\mathrm{d}s$ as $\gamma$ tends to zero. Therefore, $\left|\gamma \int^t_0 \langle  u^\gamma_r, \Delta^{k'} \varphi \rangle_{L^2}\,ds\right| \le \gamma \|\Delta^{k'} \varphi\|_{L^2} \int^t_0 \|u^\gamma_r\|_{L^2}\mathrm{d}s \to 0$.
    \item[-] Finally, we look at passage to the limit in the nonlinear term  $\int^t_0\langle \theta_r(\|\nabla u^\gamma_r\|_\infty)\vert\nabla u^\gamma_r\vert^2 , \varphi\rangle_{L^2}\mathrm{d}s$. First we show that $$\langle |\nabla u^\gamma_r|^2, \varphi \rangle_{L^2} \to \langle |\nabla u_r|^2, \varphi \rangle_{L^2}$$
in $L^\infty([0,T])$,     $\widetilde{\mathbb{P}}-a.s.$  as $\gamma$ tends to  zero. We have, 
\begin{align*}
    &\left| \langle|\nabla u^\gamma_r|^2-|\nabla u_r|^2, \varphi \rangle_{L^2}\right| \le \| \varphi\|_{L^\infty}\||\nabla u^\gamma_r|-|\nabla u_r|\|_{L^2(\mathbb{T}^n)}\||\nabla u^\gamma_r|+|\nabla u_r|\|_{L^2(\mathbb{T}^n)} \\
     &\qquad\le \|\varphi\|_{L^\infty} \left( \| | \nabla u^\gamma_r|\|_{L^2} + \|| \nabla u^\gamma_r|\|_{L^2}\right)\||\nabla u^\gamma_r|-|\nabla u_r|\|_{L^2(\mathbb{T}^n)} \to 0 
\end{align*}
as $\gamma$ tends to zero using convergence for $\{u^\gamma_r\}$ to $\{u_r\}$ in $C([0,T]; H^1(\mathbb{T}^k))$. Also, using the Lipschitz continuity of $\theta_r$,  $\theta_r(\|\nabla u^\gamma_r\|_\infty) \to \theta_r(\|\nabla u_r\|_\infty) $ as $\gamma$ tends to zero. Therefore we conclude that
\begin{align*}
 \int^t_0 \theta_r(\|\nabla u^\gamma_r\|_\infty) \langle |\nabla u^\gamma_r|^2, \varphi \rangle_{L^2}\mathrm{d}s \to \int^t_0 \theta_r(\|\nabla u_r\|_\infty) \langle |\nabla u_r|^2, \varphi \rangle_{L^2}\mathrm{d}s
\end{align*}
as $\gamma$ tends to zero.
\end{itemize}
This finishes the proof of existence of global weak solutions to \eqref{truncated}, but, in fact Eq. \eqref{truncated} is satisfied in the strong sense by $u_r$ as the paths belong to $C([0,T];H^\beta(\mathbb{T}^n))$ for $\beta \ge 2$, and one can use the embedding $H^\beta(\mathbb{T}^n)\hookrightarrow C^2(\mathbb{T}^n) $ for $\beta$ sufficiently large. We are left to prove the continuity of $u_r$ in $H^k(\mathbb{T}^n)$. From Proposition \ref{apriori}, we know that $u^\gamma_r\in L^2(\Omega;L^\infty([0,T]; H^k(\mathbb{T}^n)))$ and therefore we have $u_r \in L^2(\Omega;L^\infty([0,T];H^k(\mathbb{T}^n)))$ by Fatou's lemma. This implies that $u_r \in C([0,T];H^\beta(\mathbb{T}^n)) \cap L^\infty([0,T];H^k(\mathbb{T}^n))$ almost surely and therefore $u \in C_w([0,T];H^k(\mathbb{T}^n))$ (see  \cite[Page 263, Lemma 1.4]{temam}) almost surely. To conclude, we need to prove that the map $t \mapsto \|u_r(t)\|_{H^k}$ is continuous almost surely on $[0,T]$ and for this we follow \cite[Proposition 3.2]{alonso-non-local}. As $u_r\in L^2( \Omega; L^\infty([0,T];H^k(\mathbb{T}^n)))$, the stopping time $\tau_N=\inf\{t \ge 0 : \|u(t)\|_{H^k}>N\} \to \infty$ as $N \to \infty$ a.s.
. Thus, we only require to prove the continuity till time $\tau_N \wedge T$ for each $N \ge 1$. Similar to Eq. \eqref{Hk-estimate}, we have the following estimate 
\begin{align*}
\begin{aligned}
 \| u_r(t)\|^2_{H^k} &\lesssim \|u_r(s)\|^2_{H^k}+C(V,r)|t-s|+ C(r)  \int^t_s \|u_r(\tau)\|^2_{H^k}\mathrm{d}\tau\\ &+\sum^n_{i=1}\int_s^t\left\langle \mathscr{L}_i u_r(\tau),  u_r(\tau)\right\rangle_{L^{2}} \mathrm{d} W_i(\tau)+\sum^n_{i=1}\int_s^t\left\langle\Lambda^{k} \mathscr{L}_i u_r(\tau), \Lambda^k u_r(\tau)\right\rangle_{L^{2}} \mathrm{d} W_i(\tau)
\end{aligned}
\end{align*}
Using the fact that $\|u(t)\|_{H^k}<N$ on $[0, \tau_N \wedge T]$ and proceeding similarly as for Eqn \eqref{squaring-for-BDG}, we get
\begin{align*}
  E\left[ \left(\|u_r(t \wedge \tau_N)\|^2_{H^k}-\|u_r( s \wedge \tau_N)\|^2_{H^k} \right)^2\right] \le C(N,V,r,a,b)|t-s|,  
\end{align*}
almost surely, then using Kolmogorov's continuity theorem we obtain the continuity of the map $t \mapsto \| u(t \wedge \tau_N)\|_{H^k}$ as desired.
%%%%%%%%%%%%%%%%%%%%%%%%%%%%%%%
\section{Proof of the Main Result}\label{sec:main}
With the well-posedness result of \eqref{truncated} at hand, we are now in a position to provide the proof of the main result of this article viz. Theorem \ref{uniqueness}. We start with establishing the uniqueness result.
\subsection{Uniqueness result}
In the following proposition, we show the local uniqueness of solutions to \eqref{HJB}. The result is proved using a contradiction argument and shows that if we start with two different solutions of \eqref{HJB} defined up to a certain stopping time then they must coincide almost surely. 
\begin{prop}\label{uniquess-local}
Let $\tau$ be a stopping time and $u_1,u_2:\Omega \times [0, \tau) \times \mathbb{T}^n$ be two $H^k(\mathbb{T}^n)$-valued solutions of \eqref{HJB} with the same initial data $v_0 \in H^k(\mathbb{T}^n)$, then $u_1 =u_2$ on $[0,\tau)$.
\end{prop}
\begin{proof}
Let the following relation be satisfied
\begin{align*}
       \mathrm{d}u_j(s,x)-\sum^n_{i=1}\mathscr{L}_i u_j(s,x)\circ\mathrm{d}W_i(s)=\left(V(x)+\mu\Delta u_j-\frac{1}{2}\vert\nabla u_j(s,x)\vert^2\right)\mathrm{d}s, \quad \text{for } j=1,2.
\end{align*}
Define $u= u_1-u_2$, then $u$ satisfies the following
\begin{align*}
    \mathrm{d}u(s,x)-\sum^n_{i=1} \mathscr{L}_i u(s,x)\circ\mathrm{d}W_i(s)=\left(\mu\Delta u-\frac{1}{2}\left(\vert\nabla u_1(s,x)\vert^2-\vert\nabla u_2(s,x)\vert^2\right)\right)\mathrm{d}s,
\end{align*}
which can also be written as 
\begin{align*}
    \mathrm{d}u(s,x)-\sum^n_{i=1} \mathscr{L}_i u(s,x)\mathrm{d}W_i(s)=\left(\frac12\sum^n_{i=1} \mathscr{L}^2_i u+\mu\Delta u-\frac{1}{2}\left(\nabla u_1\cdot \nabla u+ \nabla u\cdot \nabla u_2\right)\right)\mathrm{d}s.
\end{align*}
We now apply It\^o's formula for $\|\cdot\|^2_{L^2}$   to deduce
\begin{align*}
&\mathrm{d}\|u\|^2_{L^2}+\left(\langle \nabla u_1\cdot \nabla u, u\rangle_{L^2}+\langle \nabla u_2\cdot \nabla u , u \rangle_{L^2}\right)\mathrm{d}s- 2 \sum^n_{i=1}\langle \mathscr{L}_i u, u\rangle_{L^2} \mathrm{d}W_i(s)\\
     &\qquad= \left(2\mu\langle \Delta u, u\rangle+\sum^n_{i=1}\left( \langle \mathscr{L}_i u(s),\mathscr{L}_i u(s)\rangle_{L^2}+ \langle \mathscr{L}^2_i u(s), u(s)\rangle_{L^2} \right)\right)\mathrm{d}s.
\end{align*}
Using \eqref{IBP-unique} and Lemma \ref{operator} we  obtain
\begin{align*}
    \mathrm{d}\|u\|^2_{L^2} - 2 \sum^n_{i=1}\langle \mathscr{L}_i u, u\rangle_{L^2} \mathrm{d}W_i(s)\lesssim \left(1+\|u_1\|_{H^k}+\|u_2\|_{H^k}\right) \|u\|^2_{L^2}\mathrm{d}s.
\end{align*}
Denoting  $X_t= \int^t_0(\left(1+\|u_1\|_{H^k}+\|u_2\|_{H^k}\right)\mathrm{d}s$ and applying Gr\"{o}nwall's inequality we get 
\begin{align*}
exp(X_t)\|u(t)\|^2_{L^2} \lesssim \sum^n_{i=1}\int^t_0 exp(X_s) \langle  \mathscr{L}_i u, u\rangle_{L^2}\mathrm{d}W_i(s).
\end{align*}
Taking expectation allows us to conclude that
$\mathbb{E} \left(exp(X_t)\|u(t)\|^2_{L^2} \right) \le 0$ and this implies $\|u(t)\|^2_{L^2}=0$ a.s. as $X_t$ is finite, thus the uniqueness follows.
\end{proof}
\begin{rem}
    The uniqueness of the maximal solution of \eqref{HJB} follows from the local uniqueness result proved in Proposition \ref{uniquess-local}; the proof is similar to \cite[Theorem 15]{crisan}. 
\end{rem}
\subsection{Existence of global solution}
First notice that the construction of local solutions of \eqref{HJB} is possible from the main existence result for the global solution of \eqref{truncated}, as illustrated in the following lemma
\begin{lemma}
Let $u_r$ be the global solution of \eqref{truncated} with the stopping time $\tau_r$ defined in Section \ref{sec:truncated}. Then $u_r$ is a local solution of $\eqref{HJB}$ in $[0,\tau_r)$.
\end{lemma}
\begin{proof}
From the definition of $\tau_r$, observe that, for $0\le t \le \tau_r$, one has $\| \nabla u_r\|_{L^\infty}\le C\| u_r\|_{H^k} \le r$, and therefore $\theta_r( \|\nabla u_r\|_\infty) =1$.    
\end{proof}
Proposition \ref{WP-truncated} implies the existence of maximal solution for \eqref{HJB}, the proof of which is similar to \cite[Theorem 14]{crisan} and can be stated as follows
\begin{thm}\label{existence-maximal}
    For given $u_0\in H^k(\mathbb{T}^n)$ and some $k>\frac{n}{2}+2$, there exists a maximal solution $(\tau_{\text{max}},u)$ of \eqref{HJB} along with the  property that either $\tau_{\text{max}}=\infty$ or $\limsup_{ t \to \tau_{\text{max}}}\|u(t)\|_{H^k(\mathbb{T}^n)}=\infty$.
\end{thm}
For our final result regarding the global existence of strong solutions we restrict our attention to the following
\begin{align}\label{HJBp}
\begin{cases}
\mathrm{d}u(s,x)=\mathscr{L} u(s,x)\circ \mathrm{d}W(s)+\left(V(x)+\mu\Delta u-\frac{1}{2}\vert\nabla u(s,x)\vert^2\right)\mathrm{d}s, \quad &\text{in } (0,T)\times \mathbb{T}^n,\\
    u(0,x)=u_0(x), & \text{on } \mathbb{T}^n,
    \end{cases}
\end{align}
where $\mathscr{L}:=a(x)\nabla+b$, s.t. $\nabla a(x)$ and $b$ are constant vectors in $\mathbb{R}^n$. Thanks to Theorem  \ref{existence-maximal}, we have existence of maximal solutions for \eqref{HJBp} and characterization of the stopping time. We establish the global existence result by proving that the blow-up in $H^k$ norm is not possible for the solutions of \eqref{HJBp}. This forces $t_\text{max}$ to take value as infinity. For this, we begin with a crucial lemma,
\begin{lemma}\label{maximum}
  Assume $\frac{\partial V}{\partial x_i} \le 0$ for each $i=1,2,\cdots,n$. Let $u$ be a solution to \eqref{HJBp}, then the following maximum principle holds for the derivative, 
    \begin{align*}
    \| \nabla u(t, \cdot) \|_{L^\infty(\mathbb{T}^n)} \le \| \nabla u(0,\cdot)\|_{L^\infty(\mathbb{T}^n)}
    \end{align*}
    for all $t \in [0,T]$.
\end{lemma}
\begin{proof}
    Notice that if $u$ solves \eqref{HJBp}, then $v=\nabla u$ satisfies the following equation
   \begin{align*}
\begin{cases}
     dv(s,x)&= (a(x)\nabla v(s,x)+( \nabla a(x)+b)v(s,x))\circ dW(s)\\
     &\qquad +\left(\nabla V(x)+\mu\Delta v(s,x)-(v\cdot \nabla v)(s,x)\right)\,ds,\quad \text{in}\, (0,T) \times \mathbb{T}^n\\
    v_i(0,x)&=\frac{\partial}{\partial x_i} u_0(x), \quad \text{on}\,(0,T)\times \mathbb{T}^n.
    \end{cases}
\end{align*}
Let us denote by $c:=\nabla a(x)+b$ and define $w(s,x):=e^{-cW(s)}v(s,x)$, then $w$ satisfies by It\^{o}'s formula (in the Stratonovich sense)
\begin{align*}
    \mathrm{d}w&=-ce^{cW(s)}v(s,x)\circ \mathrm{d}W(s)+e^{-cW(s)}\circ \mathrm{d}v(s,x)\\
    &=a(x)\nabla w(s,x)\circ \mathrm{d}W(s)+\left(e^{-cW(s)}\nabla V(x)+\mu\Delta w(s,x)-(v\cdot \nabla w)(s,x)\right)\,ds.
\end{align*}
Following the techniques in \cite{alonso}, we consider the SDE
\begin{align*}
\mathrm{d}X_t=-a(X_t)\circ \mathrm{d}W(t),
\end{align*}
and denote by $\psi_t(X_0)$ the corresponding flow with initial data $X_0\in \mathbb{T}^n$.
Applying It\^{o}-Wentzell's formula in Stratonovich form for $w$ along $X_t=\psi_t(X_0)$ gives 
\begin{align}\label{change}
\begin{aligned}
w(t,X_t)&=w(0,x)+ \int^t_0 \left(e^{-cW(s)}\nabla V(X_s)+ \mu \Delta w(s,X_s)- (v\cdot \nabla w)(s,X_s) \right)\,ds\\
&+\int^t_0 a(X_s) \nabla w(s,X_s)\circ \mathrm{d}W(s)+\int^t_0 \nabla w(s,X_s) \circ \mathrm{d}X_s,
\end{aligned}
\end{align}
with the last two terms cancelling each other a.s., so that $w$ satisfies a random PDE. 
With the change of variable $z(t,x):= w(t,\psi_t(X_0))$, observe that 
\begin{align*}
    \nabla_x w(t,X_t)&=\nabla_x \psi^{-1}(\psi_t(X_0))\nabla_{X_0} z(t,X_0),\\
    \Delta_x w(t,X_t)&= \Delta_x\psi_t^{-1}\left(\psi_t\left(X_0\right)\right) \nabla_{X_0} z\left(t, X_0\right)+\left|\nabla_x\psi_t^{-1}(\psi_t(X_0))\right|^2\Delta_{X_0} z(t, X_0)
\end{align*}
and thus \eqref{change} is equivalent to the following random PDE
\begin{align}\label{random}
    \frac{\partial z}{\partial t} + \alpha(t,X_t)\cdot \nabla_{X_0} z- e^{-cW(s)}\nabla V(X_t)= \beta(X_t) \Delta z.
\end{align}
where $\alpha= \nabla_x \psi^{-1}_t(\psi_t(X_0))v(t,X_t)-\mu \Delta_x\psi_t^{-1}\left(\psi_t\left(X_0\right)\right)$ and $\beta=\left|\nabla_x\psi_t^{-1}(\psi_t(X_0))\right|^2$.
Following \cite[Lemma 4.14]{alonso}, one concludes that $\|z(t)\|_{L^\infty(\mathbb{T}^n)^n} \le \|z(0)\|_{L^\infty(\mathbb{T}^n)^n}$ for each $t \in [0,T]$ and $1\le i \le n$, where the additional term in \eqref{random} coming from the potential $V$ can be handled after using the assumption $\frac{\partial}{\partial x_i}V \le 0$ for each $i=1,2,\cdots n$. This implies that $\|v_i(t)\|_{L^\infty(\mathbb{T}^n)} \le \|v_i(0)\|_{L^\infty(\mathbb{T}^n)}$ for each $t\in [0,T]$ and $1\le i \le n$, concluding the proof.
\end{proof}
\begin{rem}
The previous lemma plays a key role in establishing the \textit{a priori} estimates for the solution to Eq. \eqref{HJBp}. It says that the maximum principle for the derivative still holds while considering perturbation by transport type noise. This result fits well with the existing maximum principle for the solution to Burger's equation with transport noise (see \cite{alonso}) as taking derivative of HJB equation \eqref{HJB} gives the Burger's equation. 
\end{rem}
\begin{prop}
For the initial data $u_0 \in H^k(\mathbb{T}^n)$, there exists a constant $C(T)$ such that a.s.
\begin{align*}
    E \left[ \sup_{t \in [0,T]}\|u(t)\|_{H^k(\mathbb{T}^n)}\right] \le C(T).
\end{align*}
\end{prop}
\begin{proof}
Applying It\^{o}'s formula for $\frac12\|\cdot\|^2_{L^2}$ to the solution $u$ of \eqref{HJBp},
\begin{align*}
   \frac12 \|u(t)\|^2_{L^2}&= \frac12 \|u(0)\|^2_{L^2}\\
    &\quad+\int^t_0\left\langle V(x)+\mu\Delta u(s) -\frac12| \nabla u(s)|^2, u(s) \right\rangle_{L^2}\,ds\\
    &\quad +\int^t_0 \langle \mathscr{L} u(s),u(s) \rangle_{L^2} \,dW(s)+\frac12\int^t_0\left(\langle \mathscr{L}^2 u(s),u(s) \rangle_{L^2} + \langle \mathscr{L} u(s), \mathscr{L} u(s) \rangle_{L^2}\right)\,ds\\
    &\lesssim \frac12 \|u(0)\|^2_{L^2}+\int^t_0\|V\|_{L^2} \|u(s)\|_{L^2}\mathrm{d}s-\mu \int^t_0 \| \nabla u(s)\|^2_{L^2}\mathrm{d}s\\
    &+\frac12\int^t_0\||\nabla u(s)|^2\|_{L^2}\|u\|_{L^2}\mathrm{d}s+\int^t_0 \langle \mathscr{L} u(s),u(s) \rangle_{L^2} \mathrm{d}W(s)+ C\int^t_0\|u(s)\|^2_{L^2},
\end{align*}
where we used Lemma \ref{operator} to estimate the last two terms of the penultimate inequality. Next, using Lemma \ref{maximum} and $ab \le \frac{a^2+b^2}{2}$ for $a,b \in \mathbb{R}$, one gets
\begin{align}\label{ltwo}
    \|u(t)\|^2_{L^2} \lesssim \left(\|u(0)\|^2_{L^2}+\|\nabla u(0)\|^2_{L^\infty}+\|V\|^2_{L^2}\right)+ \int^t_0 \|u(s)\|^2_{L^2}\,ds+\int^t_0 \langle \mathscr{L} u(s),u(s) \rangle_{L^2} \mathrm{d}W(s).
\end{align}
% Taking sup over time and applying expectation gives 
% \begin{align*}
%     \frac12 E\left[\sup_{t\in [0,T]}\|u(t)\|^2_{L^2}\right] &+ \mu E\left[\int^T_0 \|\nabla u(s)\|_{L^2}\,ds\right] \le \frac12 \|u(0)\|^2_{L^2}+ \|V\|_{L^2}E \left[\int^T_0 \|u(s)\|_{L^2}\,ds\right]\\
%     &\qquad+E\left[\sup_{t \in [0,T]}\left| \int^t_0\sqrt{\nu} \langle \nabla u,u \rangle_{L^2} \,dW(s)\right|\right]\\
%     &\le \frac12 \|u(0)\|^2_{L^2}+ \|V\|_{L^2}E \left[\int^T_0 \|u(s)\|_{L^2}\,ds\right] +E \left[\left(\int^T_0\nu|\langle \nabla u(t),u(t) \rangle_{L^2}|^2\,dt\right)^{1/2}\right]
% \end{align*}
We now compute the evolution of $\|\Lambda^k u\|_{L^2}$ as follows
\begin{align}\label{evolution}
\begin{aligned}
    \frac12 \| \Lambda^k u(s)\|^2_{L^2}
    &=\frac12 \|\Lambda^k u(0)\|^2_{L^2}+ \int^t_0\langle \Lambda^k( \mathscr{L} \nabla u(s)), \Lambda^k u(s) \rangle_{L^2}\,\mathrm{d}W(s)\\
    & +\int^t_0 \left\langle \Lambda^k \left( -\frac{|\nabla u(s)|^2}{2}+ V(x)+\mu\Delta u(s) \right), \Lambda^k u(s)\right\rangle_{L^2}\mathrm{d}s\\
    &+\frac12 \int^t_0\left(\langle \Lambda^k(\mathscr{L}u(s) ),\Lambda^k(\mathscr{L} u(s)) \rangle_{L^2}+\langle \Lambda^k\mathscr{L}^2u(s), \Lambda^k u(s) \rangle_{L^2}\right) \mathrm{d} s 
\end{aligned}
\end{align}
The non-linear term on the right-hand side above is treated as
\begin{align}\label{nonlinear}
\left|\frac12\left\langle \Lambda^k \left(|\nabla u|^2 \right), \Lambda^k u \right\rangle_{L^2}\right| &\le \| \Lambda^k(| \nabla u |^2)\|_{L^2} \| \Lambda^k u \|_{L^2} \le C \| \Lambda^k(\nabla u )\|_{L^2}\| \nabla u \|_{L^\infty} \| \Lambda^k u \|_{L^2} \notag \\ 
& \lesssim \| \Lambda^{k+1} u\|_{L^2}\| \nabla u \|_{L^\infty} \| \Lambda^k u \|_{L^2}\notag \\
& \lesssim \frac{1}{2 \mu}\|\nabla u \|^2_{L^\infty} \| \Lambda^k u \|^2_{L^2}+ \frac{\mu}{2}\| \Lambda^{k+1} u\|^2_{L^2},
\end{align}
where we have used Lemma \ref{kato} in the second inequality and the last term on the RHS of \eqref{evolution} is written as 
\begin{align}\label{IBP}
\mu \langle \Lambda^k \Delta u, \Lambda^k u \rangle_{L^2}= - \mu \langle \Lambda^{k+1} u, \Lambda^{k+1} u \rangle_{L^2}= -\mu \| \Lambda^{k+1} u\|^2_{L^2}.
\end{align}
Using \eqref{nonlinear} and \eqref{IBP} in \eqref{evolution}, and using Lemmas \ref{maximum} and \ref{operator} we deduce that
\begin{align}\label{hk}
\begin{aligned}
      \| \Lambda^k u(t)\|^2_{L^2}
    \lesssim \left(\|\Lambda^k u(0)\|^2_{L^2}+T\|\Lambda^k V\|^2_{L^2}\right)&+\int^t_0\left( 1 +\frac{\|\nabla u(0)\|^2_{L^\infty}}{\mu}\right) \|\Lambda^k u(s) \|^2_{L^2}\mathrm{d}s\\
    &+\int^t_0\langle \Lambda^k( \mathscr{L} u(s)), \Lambda^k u(s) \rangle_{L^2}\,\mathrm{d}W(s)
    \end{aligned}
\end{align}
We can now add \eqref{ltwo} and \eqref{hk} to get
\begin{align*}
 \|u(t)\|^2_{H^k} &\lesssim C(T,V, \nabla u(0))+ \| u(0)\|^2_{H^k}+\int^t_0 \|u(s)\|^2_{H^k}\mathrm{d}s\\
 &+\int^t_0\langle \mathscr{L} u(s),  u(s) \rangle_{L^2}\,\mathrm{d}W(s)+\int^t_0\langle \Lambda^k( \mathscr{L} u(s)), \Lambda^k u(s) \rangle_{L^2}\,\mathrm{d}W(s)
\end{align*}
We proceed as after Eq. \eqref{Hk-estimate} in Proposition \ref{apriori} and conclude $$ E\left[\sup_{t \in [0,T]} \| u(t)\|^4_{H^k(\mathbb{T}^n)}\right] \le C.$$
\end{proof}
\section{Acknowledgements}
The first and second authors acknowledge the support of the FCT project UIDB/00208/2020. The third author would like to thank the FCT project CEMAPRE/REM-UIDB/05069/2020.
\appendix 
\section{Appendix}
\label{app}
Define the operator $Au= \nu \Delta^k u$, with $\nu >0$, and let $S(t)$ be the semigroup generated by $A$ in $L^2(\mathbb{T}^n)$. Then the fractional powers satisfy, for every $\alpha>0$,

\begin{itemize}
    \item[(1)] $\|u\|_{H^{2k\alpha}(\mathbb{T}^n)} \le C_\alpha \|(I-A)^\alpha u \|_{L^2(\mathbb{T}^n)}$,
    \item[(2)] $\|(I-A)^\alpha S(t)u\|_{L^2(\mathbb{T}^n)}\le \frac{C_\alpha}{ t^\alpha} \|u\|_{L^2(\mathbb{T}^n)}$,
\end{itemize}
 for all $t \in(0,T]$ and for some $C_\alpha>0$.

\begin{lemma}[\cite{alonso-leon}]\label{frac}
Let $u \in C\left([0, T] ; L^{2}\left(\mathbb{T}^n\right)\right), u_{i} \in L^2(\Omega;C\left([0, T] ; L^{2}\left(\mathbb{T}^n\right))\right), i \in \mathbb{N}$, and $t \in(0, T]$. We have
$$
\left\|\int_{0}^{t} S(t-s) u(s) \mathrm{d} s\right\|_{H^{\beta}(\mathbb{T}^n)}^2 \lesssim T^{2-\beta / k} \sup _{t \in[0, T]}\|u(s)\|_{L^{2}(\mathbb{T}^n)}^2,
$$
for $0<\beta<k$. Moreover,
$$
\mathbb{E}\left[\sup _{t \in[0, T]}\left\|\sum_{i=1}^{\infty} \int_{0}^{t} S(t-s) u_{i}(s) \mathrm{d} W_{i}(s)\right\|^2_{H^{\beta}(\mathbb{T}^n)}\right] \lesssim T^{2-\beta / k} \mathbb{E}\left[\sup _{t \in[0, T]} \sum_{i=1}^{\infty}\left\|u_{i}(t)\right\|_{L^{2}(\mathbb{T}^n)}^2\right],
$$
for $\{W_i\}_i$ independent Brownian motions and $0<\beta<k$.
\end{lemma}
\begin{lemma}[Kato-Ponce inequality \cite{kato}]\label{kato}
      Let $s>0, 1<p<\infty, f, g \in\left(H^s \cap W^{1, \infty}\right)\left(\mathbb{T}^n\right)$, $1< p_2, p_3<\infty$ and $1<p_1, p_4 \le \infty$ with
$$
\frac{1}{p_1}+\frac{1}{p_2}=\frac{1}{p_3}+\frac{1}{p_4}=\frac{1}{p} .
$$
Then the following inequality holds,
$$
\begin{array}{c}
\left\|\Lambda^s(f g)\right\|_{L^p} \lesssim\|f\|_{L^{p_1}}\left\|\Lambda^s g\right\|_{L^{p_2}}+\left\|\Lambda^s f\right\|_{L^{p_3}}\|g\|_{L^{p_4}}.
\end{array}
$$
\end{lemma}

\begin{lemma}\label{operator}
Let $\mathscr{Q}$ be a linear first order differential operator with smooth and bounded coefficients, then the following inequality holds for $u\in H^2(\mathbb{T}^n)$
\begin{align*}
    \langle \mathscr{Q}^2 u, u\rangle_{L^2} +\langle \mathscr{Q} u, \mathscr{Q} u \rangle_{L^2} \lesssim \|u\|^2_{L^2}.
\end{align*}
For $u \in H^{k+2}(\mathbb{T}^n)$ and $\mathscr{P}$ a pseudo-differential operator of order $k$ the following holds
\begin{align*}
    \langle \mathscr{P}\mathscr{Q}^2 u,\mathscr{P} u\rangle_{L^2} +\langle \mathscr{P}\mathscr{Q} u, \mathscr{P}\mathscr{Q} u \rangle_{L^2} \lesssim \|u\|^2_{H^k}.
\end{align*}
\end{lemma}
Next we state the following two classical results
\begin{thm}[Skorokhod representation \cite{daprato}] Let $\left\{\mu_n\right\}_{n \in \mathbb{N}}$ be a sequence of probability measures that converges weakly to some measure $\mu$. Assume the support of $\mu$ is separable. Then there exists a probability space $(\Omega, \mathcal{A}, \mathbb{P})$ and random variables $\left\{X_n\right\}_{n=1}^{\infty}$, such that $X_n$ converges almost surely to a random variable $X$, where the laws of $X_n$ and $X$ are $\mu_n$ and $\mu$, respectively.
\end{thm}
\begin{lemma}[Gyöngy-Krylov lemma \cite{gyongy}]\label{gyongy} Let $\left\{X_n\right\}_{n \in \mathbb{N}}$ be a sequence of random variables with values in a Polish space $(E, d)$, endowed with the Borel $\sigma$-algebra. Then $X_n$ converges in probability to an $E$-valued random process if and only if, for every pair of subsequences $\left\{X_{n_j}, X_{m_j}\right\}_{j \in \mathbb{N}}$, there exists a further subsequence that converges weakly to a random variable supported on the diagonal $\{(x, y) \in E \times E: x=y\}$.
\end{lemma}
\bibliographystyle{plain}
\bibliography{laplacian}

\end{document}